\documentclass{article}
\usepackage{amsthm,amsfonts,amscd,amsmath,comment,color,amssymb}
\usepackage[all]{xy}
\newtheorem{thm}{Theorem}[section]
\newtheorem{claim}[thm]{Claim}
\newtheorem{conj}[thm]{Conjecture}
\newtheorem{cor}[thm]{Corollary}
\newtheorem{lem}[thm]{Lemma}
\newtheorem{prop}[thm]{Proposition}

\theoremstyle{definition}
\newtheorem{conv}[thm]{Convention}
\newtheorem{defn}[thm]{Definition}
\newtheorem{rem}[thm]{Remark}

\makeatletter
  
  \@addtoreset{equation}{section}
\makeatother

\newcommand{\bproof}{\begin{proof}}
\newcommand{\eproof}{\end{proof}}
\newcommand{\benu}{\begin{enumerate}}
\newcommand{\eenu}{\end{enumerate}}
\newcommand{\be}{\begin{eqnarray*}}
\newcommand{\ee}{\end{eqnarray*}}
\newcommand{\ben}{\begin{eqnarray}}
\newcommand{\een}{\end{eqnarray}}
\newcommand{\bx}{\be\begin{xy}\xymatrix}
\newcommand{\ex}{\end{xy}\ee}

\newcommand{\mb}{\mathbb}

\newcommand{\A}{\mathbb A}

\newcommand{\F}{\mathbb F}

\newcommand{\fl}{\F_\ell}

\newcommand{\G}{\mathbb G}
\newcommand{\pp}{\mathbb P}
\newcommand{\qq}{\mathbb Q}
\newcommand{\Q}{\mathbb Q}

\newcommand{\ql}{\qq_\ell}
\newcommand{\qlbar}{\overline\qq_\ell}
\newcommand{\zz}{\mathbb Z}
\newcommand{\Z}{\mathbb Z}
\newcommand{\zl}{\zz_\ell}
\newcommand{\zln}{\zz/\ell^n\zz}

\newcommand{\zmodln}{\zz/\ell^n\zz}

\newcommand{\mc}{\mathcal}
\newcommand{\mf}{\mathcal F}
\newcommand{\mg}{\mathcal G}
\newcommand{\mh}{\mathcal H}

\newcommand{\mk}{\mathcal K}
\newcommand{\mo}{\mathcal O}
\newcommand{\mx}{\mathcal X}

\newcommand{\Kbar}{\overline K}

\newcommand{\Aut}{\mathrm{Aut}}
\newcommand{\CC}{\mathop{CC}\nolimits}

\newcommand{\dimtot}{\mathop{\dim\mathrm{tot}}\nolimits}
\newcommand{\et}{\text{\'et}}
\newcommand{\gal}{\mathrm{Gal}}
\newcommand{\gr}{\mathrm{Gr}}
\newcommand{\id}{\mathrm{id}}
\newcommand{\inverse}{^{-1}}
\newcommand{\ltensor}{\otimes^L}
\newcommand{\Modc}{\mathrm{Mod}_c}

\newcommand{\ov}{\overline}
\newcommand{\pr}{\mathrm{pr}}
\newcommand{\proj}{\mathop{\mathrm{Proj}}\nolimits}
\newcommand{\rk}{\mathop{\mathrm{rk}}\nolimits}
\newcommand{\spec}{\mathop{\mathrm{Spec}}\nolimits}
\newcommand{\Sw}{\mathop{\mathrm{Sw}}\nolimits}

\newcommand{\tr}{\mathop{\mathrm{Tr}}\nolimits}
\newcommand{\trbr}{\mathop{\mathrm{Tr}^{\mathrm{Br}}}\nolimits}
\newcommand{\trd}{\mathop{\mathrm{Tr}^{\mathrm{Br}}_\Q}\nolimits}
\newcommand{\vani}{\overleftarrow\times}

\newcommand{\Ggal}{G\times_H\gal(\Kbar/K_0)}

\newcommand{\StraitXlocal}{Let $S$ be a strictly henselian trait with 
generic point $\eta=\spec K$ and closed point $s$ 
and let $X$ be a strictly local $S$-scheme essentially of finite type with closed point $x$ lying above $s$}
\newcommand{\UopenVgalois}
{Let $j:U\to X_\eta$ be a \dopen, 
$h:V\to U$ a \Gcover}
\newcommand{\UopenVgaloisellprime}
{\UopenVgalois, 
and $\ell$ a prime number invertible on $S$}
\newcommand{\UopenVgaloisLambda}
{\UopenVgalois, 
and $\Lambda$ a finite field of characteristic $\ell$ distinct from the residual characteristic of $S$}

\newcommand{\dopen}{dense open immersion}
\newcommand{\Gcover}{Galois \'etale covering with Galois group $G$}

\title{Wild ramification, the nearby cycle complexes, \\and the characteristic cycles of $\ell$-adic sheaves}
\author{Hiroki Kato
\thanks{Graduated School of Mathematical Science, The University of Tokyo. 
\texttt{Email: hiroki@ms.u-tokyo.ac.jp}}}
\date{}
\begin{document}
\maketitle
\begin{abstract}
We prove a purely local form of a result of Saito and Yatagawa. 
They proved that 
the characteristic cycle of a constructible \'etale sheaf 
is determined by wild ramification of the sheaf along 
the boundary of 
a compactification. 
But they had to consider ramification {\it at all the points} of the compactification. 
We give a pointwise result, 
that is, we prove that 
the characteristic cycle of a constructible \'etale sheaf around a point 
is determined by wild ramification at the point. 
The key ingredient is to prove that 
wild ramification of the stalk of the nearby cycle complex of a constructible \'etale sheaf at 
a 
point 
is determined by wild ramification at the point. 

\end{abstract}
\section{Introduction}\label{intro}
Saito developed a theory of the characteristic cycles of constructible \'etale sheaves on a variety over a perfect field in \cite{S}. 
Let $k$ be a perfect field of characteristic $p$ and $X$ be a smooth scheme purely of dimension $n$ over $k$. 
For a constructible $\Lambda$-sheaf $\mf$, where $\Lambda$ is a finite field, 
he defined an $n$-cycle $\CC(\mf)$ on the cotangent bundle $T^*X$. 
To define the characterictic cycle, 
he used the singular support $\SS(\mf)$ defined by Beilinson \cite{B}, 
which is a closed conical subset of $T^*X$ of purely of dimension $n$. 
The characteristic cycle $\CC(\mf)$ 
is defined as a unique cycle supported on $\SS(\mf)$ 
which computes, as follows, the total dimension of the vanishing cycle complex 
$R\phi_u(\mf|_W,f)$ 
for a function $f:W\to \A^1_k$ on an open subscheme $W\subset X$ 
with an isolated characteristic point $u\in W$ with respect to $\SS(\mf)$; 
\begin{eqnarray}\label{Milnor}
-\dimtot R\phi_u(\mf|_W,f)
=(\CC(\mf),df)_{T^*W,u},
\end{eqnarray}
where the right hand side is the intersection multiplicity at the point over $u$ \cite[Theorem 5.9, Definition 5.10]{S} 
(in this explanation we assumed for simplicity that $k$ is an infinite field). 
This equality is called the Milnor formula. 

In \cite[Theorem 0.1]{SY}, Saito and Yatagawa proved that 
the characteristic cycle of a constructible \'etale sheaf 
is determined by wild ramification of the sheaf along a compactification. 
More precisely, 
suppose that we have a smooth variety $U$ over $k$ with a smooth compactification $X$ 
and an $\Lambda$-sheaf $\mg$ and an $\Lambda'$-sheaf $\mg'$ on $U_{\et}$, 
for finite fields $\Lambda$ and $\Lambda'$ of characteristics distinct from that of $k$, 
which are locally constant constructible. 
Then we have $\CC(j_!\mg)=\CC(j_!\mg')$ 
if $\mg$ and $\mg'$ ``have the same wild ramification'', 
where $j$ is the open immersion $U\to X$.

We note that they have to see ramification {\it at all points} of the compactification 
to deduce the equality of the characteristic cycles. 
This is because their proof relies on global results, 
such as (a variant of) a result of Deligne and Illusie \cite[Th\'eor\`eme 2.1]{I} which states that 
if $\mg$ and $\mg'$ have the same wild ramification, then they have the same Euler characteristics; 
$\chi_c(U_{\bar k},\mg)=\chi_c(U_{\bar k},\mg')$. 
However, since both of ramification and characteristic cycles are of local nature, 
it is more preferable to work at each point. 
Here is a pointwise definition of ``having the same wild ramification'', 
which does not seem to be given in the literature: 
Let $x\in X$ be a point and 
$X_{(\bar x)}$ denotes the strict henselization at a geometric point $\bar x$ over $x$. 
We say $\mg$ and $\mg'$ {\it have the same wild ramification at} $x$ if 
$\mg|_{U\times_XX_{(\bar x)}}$ and $\mg'|_{U\times_XX_{(\bar x)}}$ 
have the same wild ramification in the following sense 
(it is equivalent to having the same wild ramification over $X_{(\bar x)}$ in the sense of Definition \ref{swr} in the text). 
\begin{defn}[{c.f.\ \cite{I}, \ \cite[Definition 2.2.1]{V}, and \cite[Definition 5.1]{SY}}]
\label{pointwise}
Let $X$ be the spectrum of a strictly henselian normal local ring 
and $U$ be a dense open subscheme of $X$. 
Let $\Lambda$ and $\Lambda'$ be finite fields of characteristics invertible on $X$ 
and let $\mg$ and $\mg'$ be an $\Lambda$-sheaf and an $\Lambda'$-sheaf respectively on $U_{\et}$ 
which are locally constant and constructible. 

We say $\mg$ and $\mg'$ {\it have the same wild ramification} 
if there exists a proper $X$-scheme $X'$ which is normal and contains $U$ as a dense open subscheme 
such that we have 
\begin{eqnarray*}
\dim_{\Lambda}(\mg_y)^\sigma
=
\dim_{\Lambda'}(\mg'_y)^\sigma
\end{eqnarray*}
for every $\sigma\in\pi_1(U,y)$, with $y$ being some geometric point, 
which is of $p$-power order  and ``ramified'' at some point on $X'$ 
(i.e., has a fixed geometric point on the normalization of $X'$ in $V$ for every Galois cover $V$ of $U$). 
\end{defn}
Here 
we emphasize that we take a modification $X'$ of $X$ to get a reasonable definition. 
A naive definition of ``having the same wild ramification'' 
without taking modifications $X'$ of $X$ is unreasonably strong (see Remark \ref{unreasonable} and Section \ref{ex}). 
This is essentially because the inertia group in higher dimension can be too big.

Our main result is the following: 
\begin{thm}[Corollary \ref{swr implies same cc}]
\label{swr same cc for lisse}
Let $X$ be a smooth variety, $j:U\to X$ be an open immersion, 
$\Lambda$ and $\Lambda'$ be finite fields of characteristics distinct from that of $k$, 
and let $\mg$ and $\mg'$ be an $\Lambda$-sheaf and an $\Lambda'$-sheaf respectively on $U_{\et}$ 
which are locally constant and constructible. 
Let $x\in X$ be a point. 
We assume that 
$\mg$ and $\mg'$ have the same wild ramification at $x$. 
Then there exists an open neighborhood $W\subset X$ of the point $x$ such that 
$\CC(j_!\mg|_W)=\CC(j_!\mg'|_W)$. 
\end{thm}

To prove Theorem \ref{swr same cc for lisse}, 
we show a general property of the nearby cycle functor, which itself is interesting;
\begin{thm}[Theorem \ref{rpsi have swr}]
\label{rpsi lisse}
Let $S$ be a henselian trait with generic point $\eta=\spec K$ and closed point $s$.  
Let $X$ an $S$-scheme of finite type 
and $j:U\to X_\eta$ an open immersion. 
Let $\Lambda$ and $\Lambda'$ be finite fields of characteristics invertible on $S$ 
and let $\mg$ and $\mg'$ be an $\Lambda$-sheaf and an $\Lambda'$-sheaf respectively on $U_{\et}$ 
which are locally constant and constructible. 
Let $x\in X_s$ be a point of the special fiber and $\bar x$ be a geometric point over $x$. 
We assume that 
$\mg$ and $\mg'$ have the same wild ramification at $x$. 
Then the stalks 
$R\psi(j_!\mg)_{\bar x}$ and $R\psi(j_!\mg')_{\bar x}$ 
viewed as virtual representations of the inertia group of $K$, 
have the same wild ramification, i.e, 
for every $\sigma$ in the wild inertia group, we have 
\be
\sum_i(-1)^i\dim_\Lambda(R^i\psi(j_!\mg)_{\bar x})^\sigma
=\sum_i(-1)^i\dim_{\Lambda'}(R^i\psi(j_!\mg')_{\bar x})^\sigma.
\ee 
\end{thm}
Note that we do not assume that $S$ is of equal-characteristic. 
%
Since the characteristic cycle is defined using the Milnor formula (\ref{Milnor}), 
it is straightforward to deduce Theorem \ref{swr same cc for lisse} from Theorem \ref{rpsi lisse}. 

We also give a variant, due to Takeshi Saito, on a relation with restrictions to curves: 
\begin{cor}[Saito, Corollary \ref{curve}]
\label{curve lisse}
Let $X$ be a smooth variety, $j:U\to X$ be an open immersion, 
$\Lambda$ and $\Lambda'$ be finite fields of characteristics distinct from that of $k$, 
and let $\mg$ and $\mg'$ be an $\Lambda$-sheaf and an $\Lambda'$-sheaf respectively on $U_{\et}$ 
which are locally constant and constructible. 
Let $x$ be a closed point of $X$. 
Assume that $\mg$ and $\mg'$ have the same rank, 
and that for every morphism $g:C\to X$ from a smooth curve $C$ with a closed point $v\in C$ lying over $x$, 
we have an equality 
$\Sw_v(g^*j_!\mg)=\Sw_v(g^*j_!\mg')$. 
Then there exists an open neighborhood $W\subset X$ of the point $x$ such that 
$\CC(j_!\mg|_W)=\CC(j_!\mg'|_W)$. 
\end{cor}

We note that it is elementary to see that the assumption in Corollary \ref{curve lisse} is weaker than that in Theorem \ref{swr same cc for lisse}, 
i.e, if $\mg$ and $\mg'$ have the same wild ramification at $x$, then they have the same rank and the same Swan conductor at $x$ after restricting to any curve. 
We show that these conditions are in fact equivalent (Theorem \ref{swr sc var}). 
The equivalence of these conditions is the main theme of \cite{K}, 
in which the case of sheaves on a surface is proved using (a known case of) resolution of singularities 
(\cite[Theorem 3.2]{K}). 
We instead use purely inseparable local uniformization due to Temkin, 
which is a suggestion by Haoyu Hu. 
Let us note that Corollary \ref{curve lisse} is a refinement of \cite[Corollary 4.7]{K}, 
where we had to consider ramification at all points on a compactification. 

Before this article was written, Haoyu Hu taught the author a proof of this result in the case of rank 1 sheaves on a surface. 
His proof relies on ramification theory and is quite different from the proof in this article. 
In an earlier version of this article, 
the corollary above was stated only in the case of sheaves on a surface, 
because the author did not know how to prove that the assumption in Theorem \ref{swr same cc for lisse} and that in Corollary \ref{curve lisse} are equivalent in general. 
Takeshi Saito pointed out that we can remove the assumption on dimension in the statement in the earlier version, 
using the theory of nearby cycle complexes over general bases. 
His proof relies on Theorem \ref{rpsi lisse} and the equivalence of the two conditions in the case of surfaces, 
but not on the equivalence in the general case. 

We explain the idea of the proof of Theorem \ref{rpsi lisse}, which is the main part of this paper. 
For that purpose, we briefly recall an argument in \cite{I}. 
Let $U$ be a normal variety over an algebraically closed field 
and let $\mg$ and $\mg'$ be an $\Lambda$-sheaf and an $\Lambda'$-sheaf respectively on $U_{\et}$ 
which are locally constant and constructible. 
In \cite{I}, it is proved that 
if $\mg$ and $\mg'$ have the same wild ramification, then we have 
$\chi_c(U,\mg)=\chi_c(U,\mg')$ \cite[Theorem 2.1]{I},
\footnote{
Though their formulation of ``having the same wild ramification'' is slightly stronger than ours, 
we can apply the same proof. 
For the detail, see \cite[Section 5]{SY}. 
} 
in the following way. 
First, taking a Galois \'etale cover $V\to U$ 
such that $\mg|_V$ is a constant sheaf, which we denote by $M$ and view as a representation of the Galois group $G$, 
they deduced the intertwining formula 
\begin{eqnarray}\label{DI int}
\chi_c(U,\mg)=\frac{1}{|G|}\sum_{g\in G}\tr(g,R\Gamma_c(V,\ql))\cdot\trbr(g,M),
\end{eqnarray}
from a canonical isomorphism 
$
R\Gamma_c(U,\mg)
\cong
R\Gamma^G(
R\Gamma_c(V,\Lambda)\otimes^LM
)$ 
and ``projectivity'', i.e., the fact that 
$R\Gamma_c(V,\Lambda)$ 
is a perfect complex of $\Lambda[G]$-modules. 
They proved that, in the intertwining formula (\ref{DI int}), 
only terms with $g$ being ``wildly ramified'' contribute, 
more precisely, if 
$\trbr(g,R\Gamma_c(V,\Lambda))$ 
is nonzero, 
then the following two conditions hold:
\begin{enumerate}
\item\label{g has fixed pt}
$g$ has a fixed point in every $G$-equivariant compactification of $V$,
\item\label{g is wild}
$g$ is of $p$-power order. 
\end{enumerate}
The assertion 1 follows from 
the Lefschetz fixed point formula. 
The key ingredients for 2 are $\ell$-independence of the trace and the ``projectivity'': 
The ``projectivity'' of 
$\trbr(g,R\Gamma_c(V,\zl))$ 
as a $\zl[G]$-complex implies, by Brauer theory, 
that the trace of an $\ell$-singular $g$ is zero, 
while it is independent of $\ell$, hence nonzero only for $g$ of $p$-power order. 


By a similar argument, Vidal proved that 
having the same wild ramification is preserved by the direct image \cite[Th\'eor\`eme 0.1]{Vnodal} 
(see also \cite{V}, \cite{Yat}). 
We prove Theorem \ref{rpsi lisse} by a similar argument with the Euler characteristic replaced by the nearby cycle complex. 

We deduce, in Section \ref{general properties}, 
an intertwining formula for $R\psi(j_!\mg)_{\bar x}$; for an element $\sigma$ of the wild inertia group,  
\begin{eqnarray}\label{int}
\trbr(\sigma,(R\psi j_!\mg)_{\bar x})
=\frac{1}{|G|}\sum_{g\in G}
\tr((g,\sigma),(R\psi(j_!h_{*}\ql))_{\bar x})
\cdot
\trbr(g,M)
\end{eqnarray}
from a similar canonical isomorphism 
$R\psi j_!\mg
\cong
R\Gamma^G(R\psi(j_!h_{*}\Lambda)\otimes_\Lambda^LM), 
$ 
and the projectivity of $(R\psi(j_!h_{*}\ql))_{\bar x}$, 
where $h$ is the finite \'etale morphism $V\to U$.  
To do this, we need some Brauer theory 
for profinite groups established in \cite{V} and recalled in Section \ref{Brauer}.


We prove that 
only terms with $g$ being ``wildly ramified'' contribute in our intertwining formula (\ref{int}). 
To this end, we establish, in Section \ref{fixed}, 
$\ell$-independence of the trace 
$\tr((g,\sigma),(R\psi(j_!h_{*}\ql))_{\bar x})$ 
and existence of a fixed geometric point when the trace is nonzero. 
These are proved at the same time 
by an argument similar to that in \cite{O} or \cite{V} (c.f., \cite{SaTweightspec}, \cite{M}, \cite{Katoell}). 

We briefly recall an argument in \cite{O}. 
In \cite[Theorem 2.4]{O}, Ochiai proved an $\ell$-independence result for a variety over a local field, 
more precisely, for a variety $X$ over a strictly henselian discrete valuation field $K$ 
and for an element $\sigma$ of the absolute Galois group of $K$, 
the alternating sum 
$\sum_q(-1)^q\tr(\sigma,H_c^q(X_{\overline K},\ql))$ 
is an integer independent of $\ell$ distinct from the residual characteristic of $K$.

When $X$ is smooth proper and has potentially semi-stable reduction, 
the weight spectral sequence of Rapoport and Zink \cite{RZ} 
for a strictly semi-stable model of a scalar extension of $X$ 
gives a geometric interpretation of the alternating sum, 
and then the $\ell$-independence follows from the Lefschetz trace formula. 
Since we do not know if every smooth proper variety has potentially semi-stable reduction, 
Ochiai used, instead of potential semi-stability, 
the existence of a semi-stable alteration due to de\ Jong and argued by induction on $\dim X$. 

To make his strategy work in our setting, 
we establish 
an analogue of the weight spectral sequence for cohomology with coefficient in the nearby cycle complex 
in Section \ref{weight}. 
This is similar to an analogue of the weight spectral sequence given by Mieda in \cite{M}. 


We recall the notions of having the same wild ramification and having universally the same conductors 
for two \'etale sheaves in Section \ref{compare}. 
The latter means, roughly speaking, having the same ranks and the same conductors after restricting to any curve. 
We also recall the main result of \cite{K} stating that the two notions are equivalent for sheaves on a ``surface''. 
Further we remove the assumption on dimension in the case where the schemes in consideration are algebraic varieties. 
We deduce Theorem \ref{rpsi lisse} in Section \ref{swr section} 
and Theorem \ref{swr same cc for lisse} with its consequences in Section \ref{cor}. 
We include Takeshi Saito's proof of Corollary \ref{curve lisse} in Section \ref{another pf}. 
Section \ref{ell adic} is a complement on a result on 
$\ell$-adic systems of complexes 
which is used in \cite{V} and \cite{Vnodal}, but does not seem to be written explicitly. 
Finally, in Section \ref{ex}, we include an example, due to Takeshi Saito, 
explaining that we need to consider modifications in the definition of ``having the same wild ramification''.

\paragraph{Acknowledgment}
The author would like to thank his advisor Takeshi Saito 
for reading the manuscript carefully, 
for giving the numerous helpful comments, 
and for showing how to remove the assumption on dimension in an earlier version of Corollary \ref{curve lisse}. 
He also thanks Haoyu Hu for suggesting to use Temkin's result to prove Theorem \ref{swr sc var} and for sharing a proof of Theorem \ref{curve lisse} in the case of rank 1 sheaves on a surface. 
The research was supported by the Program for Leading Graduate Schools, MEXT, Japan and also by JSPS KAKENHI Grant Number 18J12981.


\paragraph{Conventions}
\begin{itemize}
\item
Let $S$ be a noetherian scheme and $X$ an $S$-scheme separated of finite type. 
A {\it compactification} of $X$ over $S$ means a proper $S$-scheme containing $X$ as a dense open subscheme. 
By Nagata's compactification theorem, a compactification of $X$ over $S$ exists. 
\item
Let $X$ be a noetherian scheme and 
$\mf$ a complex (of \'etale sheaves) of $\Lambda$-modules on $X$. 
We say $\mf$ is {\it constructible} if 
the cohomology sheaves $\mh^q(\mf)$ are constructible for all $q$ and zero except for finitely many $q$. 
\end{itemize}

\section{Preliminary on Brauer theory}
\label{Brauer}
We first recall some notions from \cite{V}. 
\begin{defn}[{\cite[1.1]{V}}]\label{proj}
Let $\Lambda$ be a commutative ring 
and $G$ a pro-finite group. 
A {\it continuous $\Lambda[G]$-module} 
is a $\Lambda$-module $M$ 
equipped with an action of $G$ 
such that the stabilizer of each element of $M$ is an open subgroup of $G$. 
A morphism of continuous $\Lambda[G]$-modules 
is a homomorphism of $\Lambda$-modules which is compatible with the actions of $G$. 
We denote 
the category of continuous $\Lambda[G]$-modules 
by $\Modc(\Lambda[G])$. 
\end{defn}
\begin{conv}
\begin{enumerate}
\item\label{proj cont}
For a continuous $\Lambda[G]$-module $M$, 
we say $M$ is {\it projective} if 
it is projective in the category $\Modc(\Lambda[G])$. 
\item
For a bounded above complex $K$ of continuous $\Lambda[G]$-modules, 
we say $K$ is {\it perfect} 
if it is quasi-isomorphic to a bounded complex $P$ of continuous $\Lambda[G]$-modules 
with each $P^i$ projective in the sense of \ref{proj cont} and finitely generated over $\Lambda$. 
\end{enumerate}
\end{conv}

Here is a characterization of projectivity for continuous $\Lambda[G]$-modules;
\begin{lem}[{\cite[1.2]{V}}]
\label{char of proj}
Let $M$ be a continuous $\Lambda[G]$-module finitely generated over $\Lambda$. 
Then $M$ is projective 
if and only if 
there exists an open normal subgroup $H$ of $G$ whose supernatural order is invertible in $\Lambda$ 
which acts trivially on $M$ 
such that $M$ is projective as a $\Lambda[G/H]$-module. 
\end{lem}

As this characterization suggests, this notion of projectivity is useful only for pro-finite groups of the following type. 
\begin{defn}[{\cite[1.2]{V}}]
\label{alm prime-to-ell}
Let $\Lambda$ be a commutative ring and $G$ be a pro-finite group. 
We say $G$ is {\it admissible} relatively to $\Lambda$ 
if there exists an open subgroup of $G$ whose surnatural order is invertible in $\Lambda$. 
\end{defn}

The classical Brauer theory for representations of finite groups as in \cite{Se} 
is extended to the above framework. 
We recall a part of it from \cite{V}. 

We denote
by 
$K_\cdot(\Lambda[G])$ 
(resp.\ $K^\cdot(\Lambda[G])$)
the Grothendieck group of 
the category of continuous $\Lambda[G]$-modules (resp.\ projective continuous $\Lambda[G]$-modules) 
finitely generated over $\Lambda$. 

Let $\Lambda$ be a field of characteristic $\ell>0$ 
and $G$ a profinite group admissible relatively to $\Lambda$. 
For a continuous $\Lambda[G]$-module $M$ 
and an element $g\in G$ which is $\ell$-regular, i.e., the supernatural order of $g$ is prime-to-$\ell$, 
the Brauer trace $\trbr(g,M)$ is defined to be the sum 
$\sum_\lambda[\lambda]$, 
where $\lambda$ runs over all eigenvalues of $g$ acting on $M$ 
and $[\cdot]$ denotes the Teichmular character $\Lambda^\times\to E^\times$ into a big enough $\ell$-adic field $E$.  
When $g\in G$ is not $\ell$-regular, we define $\trbr(g,M)$ to be zero. 

Let $\mo$ be a complete discrete valuation ring which is mixed characteristic with residue field $\Lambda$ 
and maximal ideal $\mathfrak m$. 
Let $E$ denote the field of fractions of $\mo$. 
Then we have natural maps 
\bx{
K^\cdot(\Lambda[G])&\ar[l]_{\mod\mathfrak m}K^\cdot(\mo[G])\ar[r]^{e}& K_\cdot(E[G])
}\ex
defined by $\otimes_\mo\Lambda$ and $\otimes_\mo E$ respectively. 
These maps have the same properties as in the classical case (c.f. Corollary 3 in Chapter 14 and Theorem 36 in Chapter 16 of \cite{Se}):


\begin{lem}[{\cite[1.3]{V}}]
\label{ell sing}
\begin{enumerate}
\item
The map $\text{mod } \mathfrak m$ is an isomorphism. For every element $a\in K^\cdot(\mo[G])$, we have 
$\tr(g,a)=\trbr(g,a\text{ \rm{mod} }\mathfrak m)$.  
\item
Let $a$ be an element of $K_\cdot(E[G])$. 
Then $a$ lies in the image of the map $e:K^\cdot(\mo[G])\to K_\cdot(E[G])$ 
if and only if 
we have $\tr(g,a)=0$ for every $g\in G$ of supernatural order divisible by $\ell$. 
\end{enumerate}
\end{lem}

We will use the following lemma to establish an ``intertwining formula'' for the nearby cycle functor. 
\begin{lem}[{\cite[1.3]{V}, c.f. \cite[1.4.7]{I}}]
\label{proj times mod}
Let $\Lambda$ be a finite field of characteristic $\ell$ 
and $G$ be a pro-finite group which is admissible relatively to $\Lambda$. 
Let $K$ be a finite normal subgroup of $G$. 
Then, for $(a,b)\in K^\cdot(\Lambda[G])\times K_\cdot(\Lambda[G])$ and $g\in G/K$, 
we have 
\begin{eqnarray*}
\trbr(g,(ab)^K)=\frac{1}{|K|}\sum_{k\in K}\trbr(\tilde gk,a)\cdot\trbr(\tilde gk,b),
\end{eqnarray*}
where $\tilde g\in G$ is a representative of $g$. 
\end{lem}

To define the notion of having the same wild ramification, we will use the ``rational Brauer trace'' defined as follows. 
\begin{defn}[{c.f, \cite[Section 4]{SY}, \cite[Section 1]{Yat}}]\label{trd}
Let $\Lambda$ be a finite field 
and $G$ be a pro-finite group which is admissible relatively to $\Lambda$. 
Let $M$ be an element of $K^\cdot(\Lambda[G])$. 
For $g\in G$, we take a finite extension $E$ of $\Q$ containing $\trbr(g,M)$. 
We define the rational Brauer trace 
$\trd(g,M)$ by 
\be
\trd(g,M)=\frac{1}{[E:\Q]}\tr_{E/\Q}\trbr(g,M). 
\ee
It is independent of the choice of $E$. 
\end{defn}

\begin{lem}
[{c.f.\ \cite[Lemma 4.1]{SY}, \cite[Lemma 1.1, 1.2]{Yat}}]
\label{formula trd}
Let $\Lambda$ be a finite field, 
$G$ be finite group, 
and $p$ be a prime number invertible in $\Lambda$. 
Then, for $M\in K_\cdot(\Lambda[G])$ 
and $g\in G$ of $p$-power order, 
we have equalities
\be
\trd(g,M)&=&\frac{p\cdot\dim_\Lambda (M)^g-\dim_\Lambda (M)^{g^p}}{p-1},
\\
\frac{\dim_\Lambda(M)^g}{p-1}
&=&\sum_{n=1}^\infty\frac{1}{p^n}\trd(g^{p^{n-1}},M), 
\ee
where in the second equality the limit is taken in the field $\mb R$ of real numbers.
\end{lem}
\begin{proof}
The first equality is nothing but \cite[Lemma 1.1]{Yat} 
and the second follows from the first. 
\end{proof}

\section{Projectivity and an intertwining formula for the nearby cycle functor}
\label{general properties}



For a henselian discrete valuation field $K$, 
we denote by $I_K$ (resp.\ $P_K$) 
the inertia (resp.\ wild inertia) subgroup of the absolute Galois group of $K$.

\begin{defn}
Let $X\to S$ be a morphism of schemes 
with $X$ strictly local, that is, the spectrum of a strictly henselian local ring. 
We say $X$ is {\it essentially of finite type }over $S$ 
if there exists an $S$-scheme $Y$ of finite type and a geometric point $y$ of $Y$ 
such that $X$ is isomorphic to the strict localization $Y_{(y)}$ over $S$. 
\end{defn}
\begin{lem}\label{projectivity}
\StraitXlocal. 
\UopenVgaloisellprime. 
\begin{enumerate}
\item
The complex $(R\psi j_!h_*\zln)_x$ 
is 
a perfect complex of continuous $\zln[G\times P_K]$-modules 
for every $n\ge1$. 
\item\label{trbr equal ell adic tr}
For every element $(g,\sigma)\in G\times P_K$, we have an equality 
\be
\trbr((g,\sigma),(R\psi j_!h_*\fl)_x)
=
\sum_i(-1)^i\tr((g,\sigma),(R^i\psi j_!h_*\ql)_x).
\ee
\end{enumerate}
\end{lem}
\begin{proof}
1. Since $h_*\zln$ is \'etale locally a constant sheaf of free $\zln[G]$-modules, 
$j_!h_*\zln$ is in $D^b_{ctf}(X_\eta,\zln[G])$. 
Since constructibility and being of finite-Torsion dimension 
are preserved by the nearby cycle functor (\cite[Th\'eor\`eme 3.2]{fin} and \cite[2.1.13]{7xiii}), 
the complex $(R\psi j_!h_*\zln)_x$ 
is 
in $D^b_{ctf}(x,\zln[G])$. 
By \cite[Proposition 10.2.1]{Fu}, it is 
a perfect complex of $\zln[G]$-modules. 
Thus the assertion follows from \cite[Lemma 1.3.2.(2)]{V}.

2. By the assertion 1, we can apply Proposition \ref{inverse limit of complexes} to 
the inverse system 
$((R\psi j_!h_*\zln)_x)_n$ 
to obtain the desired equality. 
\end{proof}

In the proof of the intertwining formula (Lemma \ref{intertwine}), we use the following elementary lemma. 
\begin{lem}\label{gp coh}
Let $h:V\to U$ be a $G$-torsor of schemes for a finite group $G$. 
Let $\mathcal G$ be a sheaf of abelian groups on $U_{\text{\'et}}$. 
\begin{enumerate}
\item
The canonical morphism $\mg\to (h_*h^*\mg)^G$ is an isomorphism. 
\item
If $h$ is \'etale, then $H^i(G,h_*h^*\mg)=0$ for $i>0$, 
and hence we have a canonical isomorphism 
$\mg\to R\Gamma(G,h_*h^*\mg)$.  
\end{enumerate}
\end{lem}
\begin{proof}
1. This follows from $(h_*h^*\mg)_x\cong\bigoplus_{y\mapsto x}\mg_x$. 

2. Since group cohomology is compatible with pullback \cite[9.1]{Fu}, 
we may assume that $U$ is the spectrum of a separably closed field. 
Then the assertion follows from the vanishing of the usual group cohomology \cite[Proposition 1 in Chapter VII.\ \S.2]{corps}. 
\end{proof}
\begin{lem}[intertwining formula]
\label{intertwine}
\StraitXlocal. 
\UopenVgaloisLambda. 
\begin{enumerate}
\item
Let $\mathcal G$ be a locally constant constructible sheaf of $\Lambda$-modules on $U_{\text{\'et}}$ 
such that the pullback $\mathcal G|_V$ is constant. 
Let $M$ be the representation of $G$ defined by $\mathcal G$. 
Then we have a canonical isomorphism 
\begin{eqnarray*}
R\psi j_!\mg
\cong
R\Gamma^G(R\psi(j_!h_{*}\Lambda)\otimes_\Lambda^LM) 
\end{eqnarray*}
in the derived category of sheaves of $\Lambda$-modules on $X_{\bar s}$ with a continuous action of $G\times I_K$. 
\item
Let $\mg$ be a locally constant constructible complex of $\Lambda$-modules on $U_{\text{\'et}}$ 
such that each $\mh^q(\mg|_V)$ is constant. 
We define a virtual representation $M$ of $G$ by 
$
M=\sum_q(-1)^q\Gamma(V,\mh^q(\mg))
$. 
Then, for every geometric point $x$ of $X_s$ and every $\sigma\in P_K$, 
we have an equality 
\begin{eqnarray}
\label{inter}
\trbr(\sigma,(R\psi j_!\mg)_x)
=\frac{1}{|G|}\sum_{g\in G}
\tr((g,\sigma),(R\psi(j_!h_{*}\ql))_x)
\cdot
\trbr(g,M),
\end{eqnarray} 
where 
$\tr((g,\sigma),(R\psi(j_!h_{*}\ql))_x)$ is defined to be 
the alternating sum 
\begin{eqnarray*}
\sum_i(-1)^i\tr((g,\sigma),(R^i\psi (j_!h_*\ql))_x).
\end{eqnarray*}
\end{enumerate}
\end{lem}

\begin{proof}
1. By Lemma \ref{gp coh}, we have a canonical isomorphism 
$\mg\cong R\Gamma^G(h_*h^*\mg)$. 
Here, since $h^*\mg$ is a constant sheaf, we have a canonical isomorphism 
$h_*\Lambda\otimes M\cong h_*h^*\mg$, 
by the projection formula. 
Hence, we have canonical isomorphisms
\begin{eqnarray*}
R\psi j_!\mg
&\cong&
R\psi(j_!R\Gamma^G(h_*\Lambda\otimes M))\\
&\cong&
R\Gamma^GR\psi(j_!(h_*\Lambda\otimes M)).
\\
&\cong&
R\Gamma^G(R\psi(j_!h_*\Lambda)\otimes^LM).
\end{eqnarray*}
Here the second isomorphism comes from 
the commutativity of $R\Gamma^G$ with $j_!$ and $R\psi$ 
and the last one the projection formula for the nearby cycle functor \cite[2.1.13]{7xiii}.
%

2. Since both sides are additive in $\mg$, 
we may assume that $\mg$ is a locally constant and constructible sheaf. 
By Lemma \ref{projectivity}.1, 
$(R\psi(j_!h_{*}\Lambda))_x$
defines a class $a\in K^\cdot(\Lambda[G\times P_K])$. 
We can apply Lemma \ref{proj times mod} to the class $a$ and 
the class $b\in K_\cdot(\Lambda[G\times P_K])$ defined by $M$. 
Then we obtain the desired equality using the assertion 1 and Lemma \ref{projectivity}.\ref{trbr equal ell adic tr}. 
\end{proof}

\section{An analogue of the weight spectral sequence}
\label{weight}
We establish 
an analogue of the weight spectral sequence of Rapoport and Zink \cite{RZ} 
for cohomology with coefficient in the nearby cycle complex. 

We first recall the monodromy filtration on the nearby cycle complex and 
the identification of its graded pieces in the strictly semi-stable case from \cite{SaTweightspec}. 
Let $Y$ be a strictly semi-stable scheme over a strictly henselian trait $S=\spec\mo_K$. 
We denote the closed point of $S$ by $s$. 
We take 
a separable closure $\Kbar$ of $K$. 
By \cite[4.5]{Gabber} (or \cite[Lemma 2.5.(1)]{SaTweightspec}), 
the nearby cycle complex $R\psi\ql$ is a perverse sheaf on the special fiber $Y_{s}$. 
We have a filtration, called the monodromy filtration, on $R\psi\ql$ in the category of perverse sheaves defined as follows. 
Let $t_\ell:\gal(\Kbar/K)\to \zl(1)$ be the canonical surjection defined by $\sigma\mapsto(\sigma(\pi^{1/\ell^m})/\pi^{1/\ell^m})_m$ 
for a prime element $\pi$ of $\mo_K$. 
We choose an element $T\in\gal(\Kbar/K)$ such that $t_\ell(T)$ is a topological generator of $\zl(1)$. 
Then the endomorphism $N=T-1$ of $R\psi\ql$ is nilpotent \cite[Corollary 2.6]{SaTweightspec} 
and induces the unique increasing finite filtration $M_\bullet$ on $R\psi\ql$ such that $NM_{\bullet}\subset M_{\bullet-2}$ and such that $N^k$ induces $\gr^M_k\cong\gr^M_{-k}$ (\cite[Proposition 1.6.1]{weilii}). 

To give the identification of the graded pieces $\gr^M_\bullet R\psi\ql$, 
we introduce some notations. 
We write the special fiber $Y_s$ as the union of its irreducible components; 
$Y_s=\bigcup_{i=1}^rD_i$. 
For a subset $I\subset\{1,\ldots,r\}$, we put $D_I=\cap_{i\in I}D_i$ 
and $Y^{(p)}=\coprod_{\substack{I\subset\{1,\ldots,r\}\\ |I|=p+1}}D_I$, for an integer $p\geq0$. 
We denote the natural morphism $Y^{(p)}\to Y_s$ by $a_{(p)}$. 
By \cite[Proposition 2.7]{SaTweightspec}, we have a canonical isomorphism 
\begin{eqnarray*}
\gr^M_pR\psi\ql
\cong
\bigoplus_{i\geq\max(0,p)}a_{(-p+2i)*}\ql(-i)[p-2i]. 
\end{eqnarray*} 
We consider a scheme $Z$ separated of finite type over a separably closed field $\Omega$ with a commutative diagram 
\begin{eqnarray}\label{equiv diagram}
\begin{xy}
\xymatrix{
Y_{s}\ar[d]&Z\ar[l]\ar[d]\\
s&\spec\Omega.\ar[l]
}
\end{xy}
\end{eqnarray}
Then, by \cite[Lemme 5.2.18]{SaM}, we get a natural 
spectral sequence 
\begin{eqnarray}\label{spec seq}
E_1^{p,q}=\bigoplus_{i\geq\max(0,-p)}H_c^{q-2i}(Z^{(p+2i)},\ql)(-i)
\Rightarrow
H_c^{p+q}(Z,R\psi\ql), 
\end{eqnarray}
where $Z^{(p)}$ is the fiber product $Y^{(p)}\times_{Y_s}Z$. 

We note that if the structural morphism $Y\to S=\spec\mo_K$ is equipped with an action of a finite group $G$ 
then the above constructions are equivariant with respect to natural actions. 
More precisely, 
let $K_0$ be the fixed subfield $K^G$ 
and write $H$ for the Galois group $\gal(K/K_0)$. 
Since $\bar j:Y_{\Kbar}\to Y$ has a natural $\Ggal$-structure, 
the nearby cycle complex $R\psi\ql=i^*R\bar j_*\ql$, where $i$ is the closed immersion $Y_s\to Y$, 
also has a natural $\Ggal$-structure. 
The monodromy filtration $M_\bullet$ on $R\psi\ql$ is a filtration by $\Ggal$-stable sub-perverse sheaves. 
Therefore, if we consider the trivial action on $\Omega$ and if the diagram (\ref{equiv diagram}) is $G$-equivariant, 
then the spectral sequence (\ref{spec seq}) is $\Ggal$-equivariant. 
Let us summarize the above argument. 
\begin{lem}\label{weight spec seq}
Let $S$ be a strictly henselian trait with generic point $\spec K$ and closed point $s$. 
We consider a $G$-equivariant commutative diagram 
$$
\begin{xy}
\xymatrix{
Y\ar[d]&Y_{s}\ar[l]\ar[d]&Z\ar[l]\ar[d]\\
S&s\ar[l]&\spec\Omega,\ar[l]
}
\end{xy}
$$
such that 
the vertical arrows are morphisms separated of finite type, 
$\Omega$ is a separably closed field, 
and the actions of $G$ on $s$ and $\Omega$ are trivial. 
We write $K_0$ for the fixed subfield $K^G$ of $K$ and $H$ for the Galois group of the extension $K/K_0$. We assume that $Y$ is strictly semi-stable over $S$ and let $Z^{(p)}$ be as above. 
Then, we have a $\Ggal$-equivariant spectral sequence 
\begin{eqnarray*}
E_1^{p,q}=\bigoplus_{i\geq\max(0,-p)}H_c^{q-2i}(Z^{(p+2i)},\ql)(-i)
\Rightarrow
H_c^{p+q}(Z,R\psi\ql),
\end{eqnarray*}
where $\Ggal$ acts on the $E_1$-terms through the surjection 
$\Ggal\to G$. 
\end{lem}

\section{$\ell$-independence and existence of a fixed point}
\label{fixed}
We prove the $\ell$-independence of the trace 
$\tr((g,\sigma),(R\psi(j_!h_{*}\ql))_{x})$ 
and the existence of a fixed geometric point when the trace is nonzero (Proposition \ref{general formulation}). 
Using these results, we prove that in the intertwining formula (\ref{inter}) 
only terms with $g$ being ``wildly ramified'' contribute (Proposition \ref{contribution}, Corollary \ref{str hens}); 
We use the existence of a fixed geometric point to deduce that 
$g$ is ``ramified'' when the trace is nonzero. 
Further the $\ell$-independence 
combined with the projectivity of the nearby cycle complexes (Lemma \ref{projectivity}) 
and the Brauer theory (Lemma \ref{ell sing}) 
implies that $g$ is ``wild'' 
(Proposition \ref{contribution}.2).

The $\ell$-independence and the existence of a fixed geometric point 
are eventually reduced to the following ``geometric'' lemma. 
\begin{lem}\label{key lemma over k}
Let $Z$ be a scheme separated of finite type over a separably closed field $\Omega$ 
and $G$ be a finite group acting on $Z$ by $\Omega$-automorphisms. 
Let $\ell$ be a prime number distinct from the characteristic of $\Omega$. 
Then for every $g\in G$ the following hold:
\begin{enumerate}
\item
The alternating sum 
\begin{eqnarray}\label{alt sum over k}
\sum_i(-1)^i\tr(g,H_c^i(Z,\ql)) 
\end{eqnarray}
is an integer independent of the prime number $\ell$ distinct from the characteristic of $\Omega$. 
\item
We take a $G$-equivariant compactification $\overline Z$ over $\Omega$. 
Assume further that the alternating sum (\ref{alt sum over k}) is nonzero. 
Then there exists a geometric point 
of $\overline Z$ 
fixed by $g$. 
\end{enumerate}
\end{lem}
This lemma seems to be well-known. 
It can be reduced to the case where $Z$ is smooth and proper over $\Omega$ 
using de\ Jong's result on alterations. 
Then it follows from the Lefschetz trace formula. 
But for the completeness, we give a reference for the lemma:
\begin{proof}
The assertions are special cases of results in \cite{V}. 
The assertion 1 (resp.\ 2) follows from 
\cite[Proposition 4.2]{V} (resp.\ \cite[Proposition 5.1]{V}); 
by putting, under the notations there, 
$S=\spec\Omega[[t]]$, 
$V=Z\times_\Omega \Omega((t))$, 
$\eta_1=\spec\Omega((t))$, 
and $H=G$ 
(resp.\ 
$S=S_1=\spec\Omega[[t]]$, 
$Y_1=\overline Z\times_\Omega\Omega[[t]]$, 
$Z_1=Z\times_\Omega \Omega((t))$, 
$H=G$). 
\end{proof}

The following is the heart of this paper.
\begin{prop}\label{general formulation}
Let $S$ be a strictly henselian trait with generic point $\eta=\spec K$ and closed point $s$. 
Let $G$ be a finite group acting on $S$. 
Assume that the action of $G$ on $s$ is trivial. 
We consider a following $G$-equivariant commutative diagram
\begin{eqnarray}\label{Z diagram}
\begin{xy}
\xymatrix{
V\ar[r]^j&Y_\eta\ar[r]\ar[d]&Y\ar[d]&Y_s\ar[l]\ar[d]&Z\ar[l]\ar[d]\\
&\eta\ar[r]&S&s\ar[l]&\spec\Omega,\ar[l]
}
\end{xy}
\end{eqnarray}
such that the vertical arrows are morphisms separated of finite type, 
$j:V\to Y$ is a dense open immersion, 
and $\Omega$ is a separably closed field with the trivial $G$-action. 
Let $\ell$ be a prime number invertible on $S$. 
We write $K_0$ for the fixed subfield $K^G$ of $K$ and $H$ for the Galois group of the extension $K/K_0$. 
Then, for every $(g,\sigma)\in \Ggal$, the following hold:
\begin{enumerate}
\item
The eigenvalues of $(g,\sigma)$ acting on $H^i_c(Z,R\psi(j_!\ql))$ are roots of the unity.
\item
The alternating sum 
\begin{eqnarray}\label{alt sum}
\sum_i(-1)^i\tr((g,\sigma),H^i_c(Z,R\psi(j_!\ql)))
\end{eqnarray}
is an integer independent of the prime number $\ell$ invertible on $S$. 
\item
We take a $G$-equivariant compactification $\overline{Z}$ of $Z$ over $\Omega$. 
We suppose that the alternating sum (\ref{alt sum}) is non-zero. 
Then there exists a geometric point of $\overline{Z}$ which is fixed by $g$. 
\end{enumerate}
\end{prop}

\begin{rem}
Mieda also obtained a similar result in \cite[Theorem 6.1.3]{M}. 
He also treated the case of an algebraic correspondence. 
In the case of an algebraic correspondence coming from an automorphism of finite order, 
\cite[Theorem 6.1.3]{M} is equivalent to 
the special case of Proposition \ref{general formulation}.2 where $V=Y$, $\spec \Omega=s$, and $Z=Y_s$. 
\end{rem}

\begin{proof}[Proof of Proposition \ref{general formulation}]
We will reduce the proof to the ``semi-stable'' case. 
First of all, by the fact that the formation of $R\psi$ is compatible with change of the trait $S$ 
(\cite[Proposition 3.7]{fin}), 
we may assume that $S$ is complete, in particular excellent (the excellentness will be needed to use a result of de\ Jong). 
We argue by induction on $d=\dim Y_\eta$ and use the following claims. 
The $d=0$ case is straightforward. 
\begin{claim}\label{can shrink}
Let $V'$ be another $G$-stable dense open subscheme of $Y_\eta$. 
Then, under the induction hypothesis, the assertions for $V'$ are equivalent to those for $V$. 
\end{claim}
\begin{proof}
We may assume that $V'=Y$. Let $T$ be the complement $Y\setminus V$. 
We denote by $i$ the closed immersion $T_\eta\to Y_\eta$. 
Then we have a $\Ggal$-equivariant long exact sequence 
\begin{eqnarray*}
\cdots\to 
H_c^{n-1}(Z,R\psi i_*\ql)
\to 
H_c^{n}(Z,R\psi j_!\ql)
\to
H_c^{n}(Z,R\psi \ql)
\to\cdots.
\end{eqnarray*}
Since we have $H_c^{n}(Z,R\psi i_*\ql)\cong H_c^{n}(Z\times_{Y}T,R\psi_{T/S} \ql)$ 
and $\dim T_\eta<\dim Y_\eta$, we get the claim by the induction hypothesis. 
\end{proof}
\begin{claim}\label{can extend the base}
Let $K'$ be a finite quasi-Galois extension of $K$ which is also quasi-Galois over $K_0$ 
and $S'$ the normalization of $S$ in $K'$ with generic point and closed point denoted by $\eta'$ and $s'$. 
We write $H'$ for the automorphism group $\Aut(K'/K_0)$ 
and $K'_0$ for the fixed subfield $(K')^{H'}$ 
and put $G'=G\times_HH'$. 
Then, the assertions for the $G$-equivariant diagram (\ref{Z diagram}) 
are equivalent to those for the $G'$-equivariant diagram 
\begin{eqnarray*}
\begin{xy}
\xymatrix{
V_{\eta'}\ar[r]^{j'}&Y_{\eta'}\ar[r]\ar[d]&Y\times_SS'\ar[d]&Y_{s'}\ar[l]\ar[d]&Z\times_{\spec\Omega}\spec\Omega'\ar[l]\ar[d]\\
&\eta'\ar[r]&S'&s'\ar[l]&\spec\Omega',\ar[l]
}
\end{xy}
\end{eqnarray*}
with $\Omega'$ being a separably closed field extension of $\Omega$. 
\end{claim}
\begin{proof}
This follows from the $\Ggal\cong G'\times_{H'}\gal(\overline{K'}/K'_0)$-equivariant isomorphism 
$R\psi_{Y/S}j_!\ql\cong R\psi_{Y'/S'}j'_!\ql$. 
\end{proof}

\begin{claim}\label{red to alt}
Assume that $Y$ is an integral scheme. 
Let $G'$ be a finite group with a surjection $G'\to G$. 
Let $Y'$ be an integral scheme and $(Y',G')\to (Y,G)$ a Galois alteration, 
i.e, a $G'$-equivariant proper generically finite surjection $Y'\to Y$
such that the fixed subfield $K(Y')^{\Gamma}$ of the function field $K(Y')$ of $Y'$ by $\Gamma=\ker(G'\to \Aut(Y))$ 
is purely inseparable over the function field $K(Y)$ of $Y$. 
We put $V'=V\times_YY'$ and $Z'=Z\times_{Y_s}Y'_s$. 
Then the assertions for the $G'$-equivariant diagram 
\begin{eqnarray}\label{altered diagram}
\begin{xy}
\xymatrix{
V'\ar[r]^{j'}&Y'_{\eta}\ar[r]\ar[d]&Y'\ar[d]&Y'_{s}\ar[l]\ar[d]&Z'\ar[l]\ar[d]\\
&\eta\ar[r]&S&s\ar[l]&\spec\Omega,\ar[l]
}
\end{xy}
\end{eqnarray}
imply those for the original diagram (\ref{Z diagram}). 
\end{claim}
\begin{proof}
By Claim \ref{can shrink} we may assume that the morphism $V'\to V$ is finite 
and that the natural morphism $V'\to V'/\Gamma$ is \'etale. 
Then, by Lemma \ref{Gamma fixed part} below, 
the assertion 1 for the original diagram (\ref{Z diagram}) follows from the assertion 1 for the altered diagram (\ref{altered diagram}). 
Further, by the same lemma, 
we have 
\begin{eqnarray*}
\tr((g,\sigma),H_c^q(Z,R\psi_{Y/S} j_!\ql))
=
\frac{1}{|\Gamma|}\sum_{g'}
\tr((g',\sigma),H_c^q(Z',R\psi_{Y'/S} j'_!\ql)), 
\end{eqnarray*}
where $g'$ runs over elements of $G'$ which induce the same automorphism of $Y$ as $g$. 
Thus, 
the assertions 2 and 3 for the original diagram (\ref{Z diagram})
 follow from those for the altered diagram (\ref{altered diagram}). 
\end{proof}
\begin{lem}
\label{Gamma fixed part}
The natural morphism 
$
H_c^i(Z,R\psi_{Y/S}(j_!\ql))
\to
H_c^i(Z',R\psi_{Y'/S}(j'_!\ql))^\Gamma
$
is an isomorphism. 
\end{lem}
\begin{proof}
In general, for an $\ell$-adic sheaf $\mf$ on a scheme, with an action of a finite group $G$ which is trivial on the scheme, 
we have a natural morphism 
$N:\mf\to\mf^G$ 
defined by $x\to\sum_{g\in G}gx$. 
We note that $N$ induces the multiplication-by-$|G|$ map on $\mf^G$. 
We will apply this observation to $\mf=h_*h^*\ql$ to get the inverse of the map in the assertion, 
where $h$ denote the natural morphism $V'\to V$. 

In the following we write simply $R\psi$ for $R\psi_{Y/S}$. 
By the proper base change theorem, 
we can identify the map in the assertion with the natural map
\be
\alpha:H_c^i(Z,R\psi(j_!\ql))\to H_c^i(Z,R\psi(j_!h_*h^*\ql))^\Gamma.
\ee
induced by $\ql\to h_*h^*\ql$. 
We define a map 
\be
\beta:H_c^i(Z,R\psi(j_!h_*h^*\ql))\to H_c^i(Z,R\psi(j_!\ql))
\ee
to be the morphism induced by $N:h_*h^*\ql\to(h_*h^*\ql)^\Gamma\cong\ql$. 
Then, by the observation in the beginning, 
$\beta\circ\alpha$ is the multiplication-by-$|\Gamma|$ map on $H_c^i(Z,R\psi(j_!\ql))$. 
Further, 
$\alpha\circ\beta$ is the map given by $x\mapsto\sum_{g\in\Gamma}gx$, 
and hence it induces the multiplication-by-$|\Gamma|$ map on $H_c^i(Z,R\psi(j_!h_*h^*\ql))^\Gamma$. 
Thus, the assertion follows. 
\end{proof}

By Claim \ref{can extend the base}, 
we may assume that every irreducible component of $Y_\eta$ is geometrically integral over $\eta$. 
By applying Claim \ref{red to alt} to the normalization $Y'\to Y$, we may assume that $Y$ is normal. 
By considering each orbit of $\langle g\rangle$ acting on the set of connected components of $Y$, 
we may assume that $\langle g\rangle$ acts transitively on the set of connected components of $Y$. 
But, then the traces are nonzero only if $Y$ is connected. 
Thus, we may assume that 
$Y$ is normal and connected and that $Y_\eta$ is geometrically integral over $\eta$. 
Then, by \cite[Proposition 4.4.1]{V}, we can find a surjection $G'\to G$ of groups and a $G'$-equivariant diagram
\begin{eqnarray}\label{alt diagram}
\begin{xy}\xymatrix{
Y\ar[d]&Y\times_SS'\ar[l]\ar[d]&Y'\ar[l]_{\phi}\ar[ld]\\
S&S'\ar[l]
}\end{xy}
\end{eqnarray}
such that 
\begin{itemize}
\item
$S'$ is a finite extension of $S$,
\item
$(Y',G')\to(Y,G)$ is a Galois alteration, 
\item
$Y'$ is strictly semi-stable over $S'$.
\end{itemize}
By Claim \ref{can extend the base} we may assume that $S'=S$. 
Further by Claim \ref{red to alt} below, we may assume that $Y$ is strictly semi-stable over $S$. 
By Claim \ref{can shrink}, we may further assume that $V=Y_\eta$. 
By Lemma \ref{weight spec seq}, we have a $\Ggal$-equivariant spectral sequence
\begin{eqnarray}\label{wss}
E_1^{p,q}=\bigoplus_{i\geq\max(0,-p)}H_c^{q-2i}(Z^{(p+2i)},\ql)(-i)
\Rightarrow
H_c^{p+q}(Z,R\psi\ql),
\end{eqnarray}
where we use the notations in Section \ref{weight}. 
Since $\Ggal$ acts on the $E_1$ terms through the surjection 
$\Ggal\to G$, 
the assertion 1 follows. 
By the spectral sequence (\ref{wss}), we obtain an equality 
\be
\sum_{\substack{p,q\\i\geq\max(0,-p)}}(-1)^{p+q}
\tr(g,H_c^{q-2i}(Z^{(p+2i)},\ql))
=
\sum_i(-1)^i\tr((g,\sigma),H_c^{i}(Z,R\psi\ql)).
\ee 
Then the assertion 2 follows from Lemma \ref{key lemma over k}.1. 
Further, if the right hand side is nonzero, then the alternating sum
\begin{eqnarray*}
\sum_i(-1)^i\tr(g,H_c^i(Z^{(p)},\ql))
\end{eqnarray*}
is nonzero for some $p\geq0$. 
We take a $G$-equivariant compactification $\overline{Z^{(p)}}$ of $Z^{(p)}$ over $\Omega$ 
with a morphism $\overline{Z^{(p)}}\to\overline Z$ extending the natural morphism $Z^{(p)}\to Z$. 
Then, by Lemma \ref{key lemma over k}.2, we find a geometric point of $\overline{Z^{(p)}}$ fixed by $g$. 
This concludes the proof. 
\end{proof}
\begin{defn}
\label{ramified}
Let $U$ be a dense open subscheme of a scheme $X$ which is normal and connected 
and $V\to U$ be a Galois \'etale covering with Galois group $G$. 
We denote by $Y$ the normalization of $X$ in $V$. 
\benu
\item
We say an element $g\in G$ is {\it ramified on} $X$ 
if there exists a geometric point $y$ of $Y$ fixed by $g$. 
\item
We say an element $g\in G$ is {\it wildly ramified on} $X$ 
if the order of $g$ is a power of a prime number $p$ 
and there exists a geometric point $y$ of residual characteristic $p$ of $Y$ which is fixed by $g$. 
\eenu
\end{defn}

\begin{prop}\label{contribution}
Let $S$ be a strictly henselian trait, $X$ an $S$-scheme of finite type, 
$j:U\to X_\eta$ a dense open immersion with $U$ being normal and connected, 
and $h:V\to U$ be a Galois \'etale covering with Galois group $G$. 
Let $(g,\sigma)\in G\times \gal(\Kbar/K)$ 
and $x$ be a geometric point of the closed fiber $X_s$. 
If $\tr((g,\sigma),(R\psi(j_!h_{*}\ql))_x)$ is nonzero, 
then 
\begin{enumerate}
\item
for any compactification $X'$ of $U$ over $X$, 
the element $g$ is ramified on $X'$ in the sense of Definition \ref{ramified}, 
\item
the order of $g$ is a power of the residual characteristic $p$ of $S$. 
\end{enumerate}
Thus, $g$ is wildly ramified on any compactification $X'$ of $U$ over $X$. 
\end{prop}

By a standard limit argument, we can deduce the following from Proposition \ref{contribution}. 
\begin{cor}\label{str hens}
\StraitXlocal. 
Let $j:U\to X_\eta$ be a dense open immersion with $U$ normal, 
$h:V\to U$ a \Gcover, 
and $\ell$ a prime number distinct from the residual characteristic of $S$. 
Then for every $(g,\sigma)\in G\times\gal(\Kbar/K)$, 
if $\tr((g,\sigma),(R\psi(j_!h_{*}\ql))_x)$ is nonzero, 
then 
$g$ is wildly ramified on any compactification $X'$ of $U$ over $X$. 
\end{cor}

\begin{rem}
The case where the compactification $X'$ is $X$ itself, 
Proposition \ref{contribution}.1 is almost trivial, 
because we have a canonical isomorphism 
\begin{eqnarray*}
(R\psi(j_!h_{*}\ql))_x\cong\bigoplus_{y\mapsto x}(R\psi(j'_!\ql))_y,
\end{eqnarray*}
where the direct sum is taken over geometric points $y$ of $Y$ lying above $x$. 
\end{rem}
\begin{proof}[Proof of Proposition \ref{contribution}]
1. Let $Y'$ be the normalization of $X'$ in $V$. 
Put $Z=Y'_s\times_{X_s}x$. 
Then, by the proper base change theorem, we have a canonical isomorphism 
$(R\psi(j_!h_{*}\ql))_x\cong R\Gamma(Z,R\psi j'_!\ql)$. 
Thus, the alternating sum 
\begin{eqnarray*}
\sum_i(-1)^i\tr((g,\sigma),H^i(Z,R\psi j'_!\ql))
\end{eqnarray*} 
is nonzero. 
Applying Proposition \ref{general formulation}.3 to the diagram 
$$
\begin{xy}
\xymatrix{
V\ar[r]^{j'}&Y'_\eta\ar[r]\ar[d]&Y'\ar[d]&Y'_s\ar[l]\ar[d]&Z\ar[l]\ar[d]\\
&\eta\ar[r]&S&s\ar[l]&x,\ar[l]
}
\end{xy}
$$
we find a geometric point of $Z$ fixed by $g$ for some $\alpha$, 
which gives a geometric point of $Y'_s$ over $x$ which is fixed by $g$.

2. 
Let $\ell'$ be a prime number distinct from $p$. 
By Lemma \ref{projectivity}.1, 
the complex $(R\psi j_!h_*\F_{\ell'})_x$ is a perfect complex of $\F_{\ell'}[G\times P_K]$-modules, 
and hence, by Lemma \ref{ell sing}.1, we can take an element $a\in K^\cdot(\Z_{\ell'}[G\times P_K])$ 
whose reduction modulo $\ell'$ is the class of $(R\psi j_!h_*\F_{\ell'})_x$ in $K^\cdot(\F_{\ell'}[G\times P_K])$. 
Further, by Lemma \ref{projectivity}.2, we have an equality 
\be
\trbr((g,\sigma),(R\psi j_!h_*\F_{\ell'})_x)
=
\sum_i(-1)^i\tr((g,\sigma),(R^i\psi j_!h_*\Q_{\ell'})_x).
\ee
By Lemma \ref{ell sing}.2, 
if the left hand side is nonzero, then the order of $g$ is prime-to-$\ell'$. 
Here, by Proposition \ref{general formulation}.2 applied to $Z=\spec\Omega=x$, the right hand side is an integer independent of $\ell'\ne p$. 
Thus, if the alternating sum is nonzero, then $g$ is of $p$-power order. 
\end{proof}

\section{Comparing wild ramification of \'etale sheaves}\label{compare}
We recall the notions of having the same wild ramification and of having universally the same conductors 
in Subsection \ref{subsec swr} and \ref{subsec sc} respectively, 
and then, in Subsection \ref{relation}, recall the main theorem of \cite{K} in Theorem \ref {swr sc}, 
which states that the two notions are equivalent for sheaves on a ``surface''. 
Further, we remove the assumption on the dimension in the case where the schemes in consideration are algebraic varieties in Theorem \ref{swr sc var}. 

\subsection{The notion of having the same wild ramification}\label{subsec swr}
The notion of having the same wild ramification is originally introduced by Deligne-Illusie 
(c.f. \cite[Th\'eor\`eme 2.1]{I}, \cite[Definition 2.2.1, Definition 2.3.1]{V}, and \cite[Definition 5.1]{SY}, and \cite[Definition 2.2]{K}). 
Here we follow \cite{K} except that 
we use the rational Brauer trace (Definition \ref{trd}) 
instead of the dimensions of fixed parts. 
We prefer the rational Brauer trace because the intertwining formula for the rational Brauer trace 
is given in the same form as that for the Brauer trace (see Lemma \ref{int trd}), 
whereas that for the dimensions of fixed parts is more complicated. 
The definition remains equivalent even after this change (Lemma \ref{swr trd}). 
\begin{defn}
\label{swr}
Let $S$ be an excellent noetherian scheme and 
$U$ an $S$-scheme separated of finite type. 
Let $\Lambda$ and $\Lambda'$ be finite fields whose characteristics are invertible on $S$. 
Let $\mf$ and $\mf'$ be constructible complexes of $\Lambda$-modules and $\Lambda'$-modules respectively on $U_{\text{\'et}}$. 
\begin{enumerate}
\item\label{smooth case}
Assume that $U$ is normal and connected 
and that 
$\mh^q(\mf)$ and $\mh^q(\mf')$ are locally constant for all $q$. 
We take a Galois \'etale covering $V\to U$ such that 
$\mh^q(\mf|_V)$ and $\mh^q(\mf'|_V)$ are constant for all $q$. 
We define a virtual representation $M$ of $G$ by 
$
M=\sum_q(-1)^q\Gamma(V,\mh^q(\mf))
$. 
We say $\mf$ and $\mf'$ {\it have the same wild ramification} over $S$ 
if there exists a normal compactification $X$ of $U$ over $S$ 
such that 
for every element $g\in G$ which is wildly ramified on $X$ 
(in the sense of Definition \ref{ramified}), 
we have $\trd(g,M)=\trd(g,M')$. 
\item
In general, 
we say $\mf$ and $\mf'$ {\it have the same wild ramification} over $S$ 
if there exists a finite partition $U=\coprod_iU_i$ 
such that each $U_i$ is a normal and locally closed subset of $U$, 
that $\mh^q(\mf|_{U_i})$ and $\mh^q(\mf'|_{U_i})$ are locally constant for all $q$, 
and that $\mf|_{U_i}$ and $\mf'|_{U_i}$ have the same wild ramification over $S$ in the sense of \ref{smooth case}. 
\end{enumerate}
\end{defn}

Lemma \ref{formula trd} implies that 
the above definition is 
compatible with that in Section \ref{intro} and \cite[Section 2]{K} (and also with \cite[Definition 5.1]{SY} when the complexes $\mf$ and $\mf'$ are sheaves): 
\begin{lem}\label{swr trd}
Under the notations and assumptions in Definition \ref{swr}.1, 
$\mf$ and $\mf'$ have the same wild ramification 
if and only if 
there exists a normal compactification $X$ of $U$ over $S$ 
such that 
for every element $g\in G$ which is wildly ramified on $X$ 
(in the sense of Definition \ref{ramified}), 
we have $\dim_\Lambda(M)^\sigma=\dim_{\Lambda'}(M')^\sigma$. 
\end{lem}


We recall a valuative criterion for having the same wild ramification, 
which we will use in the proof of Theorem \ref{swr sc var}. 

Let $\mc O_F$ be a strictly henselian valuation ring with field of fractions $F$.  
We choose a separable closure $\ov F$ of $F$. 
The wild inertia subgroup of the absolute Galois group $G_F=\gal(\ov F/F)$ of $F$ is defined to be 
the unique (pro-)$p$-Sylow subgroup of $G_F$, 
which we denote by $P_F$. 

Let $S$ be an excellent noetherian scheme, 
$U$ be a scheme separated of finite type over $S$ which is normal and connected, 
$V\to U$ be a $G$-torsor for a finite group $G$. 
We consider commutative diagrams of the form
\ben\label{val diagram}
\begin{xy}\xymatrix{
V\ar[d]&\spec\ov F\ar[l]\ar[d]\\
U\ar[d]&\spec F\ar[l]\ar[d]\\
S&\spec\mc O_F,\ar[l]
}\end{xy}\een
with $\mc O_F$ a strictly henselian valuation ring 
and $\ov F$ a separable closure of the field of fractions $F$ of $\mc O_F$. 
For each commutative diagram (\ref{val diagram}), we have a natural map $\gal(\ov F/F)\to G$. 

\begin{lem}[{\cite[Section 6]{Vnodal}, c.f.\ \cite[Lemma 2.4]{Katoell}}]\label{val ram}
For $g\in G$, the following are equivalent
\begin{enumerate}
\item
$g$ is ramified (resp.\ wildly ramified) on every compactification of $U$ over $S$,
\item
there exists a commutative diagram (\ref{val diagram}) 
and an element $\sigma\in \gal(\ov F/F)$ (resp.\ an element $\sigma\in P_F$ of the wild inertia subgroup) 
such that $g$ is the image of $\sigma$ by the natural map $\gal(\ov F/F)\to G$. 
\end{enumerate}
Further, if $U$ is regular, then the above conditions are equivalent to 
\begin{enumerate}\setcounter{enumi}{2}
\item
there exists a commutative diagram (\ref{val diagram}) such that the image of $\spec F$ in $U$ is the generic point 
and an element $\sigma\in \gal(\ov F/F)$ (resp.\ an element $\sigma\in P_F$ of the wild inertia subgroup) such that $g$ is the image of $\sigma$ by the natural map $\gal(\ov F/F)\to G$. 
\end{enumerate}
\end{lem}
%
%
%
%
%
%

Let $S$ be an excellent noetherian scheme 
and $U$ be a scheme separated of finite type over $S$ which is normal and connected. 
We consider commutative diagrams of the form
\ben\label{val diagram 2}
\begin{xy}\xymatrix{
&\ov \eta=\spec\ov F\ar[d]\\
U\ar[d]&\spec F\ar[l]\ar[d]\\
S&\spec\mc O_F,\ar[l]
}\end{xy}\een
with $\mc O_F$ a strictly henselian valuation ring 
and $\ov F$ a separable closure of the field of fractions $F$ of $\mc O_F$. 
Lemma \ref{val ram} immediately implies the following valuative criterion for having same wild ramification.
\begin{lem}\label{val criterion}
Let $\mf$ and $\mf'$ be constructible complex of $\Lambda$-modules and $\Lambda'$-modules on $U_{\et}$ respectively 
such that $\mc H^q(\mf)$ and $\mc H^q(\mf')$ are locally constant for every $q$. 
Then the following are equivalent; 
\begin{enumerate}
\item
$\mf$ and $\mf'$ have the same wild ramification over $S$, 
\item
for every commutative diagram (\ref{val diagram 2}) 
and for every element $\sigma\in P_F$ of the wild inertia subgroup, 
we have 
\be
\trd(\sigma,\mf_{\ov\eta})=\trd(\sigma,\mf'_{\ov\eta}).
\ee

\end{enumerate}
Further, if $U$ is regular, then the above conditions are equivalent to 
\begin{enumerate}\setcounter{enumi}{2}
\item
for every commutative diagram (\ref{val diagram 2}) such that the image of $\spec F$ in $U$ is the generic point, 
for every generic geometric point $\ov \eta$, 
and for every element $\sigma\in P_F$ of the wild inertia subgroup, 
we have 
\be
\trd(\sigma,\mf_{\ov\eta})=\trd(\sigma,\mf'_{\ov\eta}).
\ee
\end{enumerate}
\end{lem}

We recall some elementary properties of the notion of having the same wild ramification, 
which follow immediately from the valuative criterion. 
\begin{lem}\label{list swr}
Let $S$ be an excellent noetherian scheme, $U$ an $S$-scheme separated of finite type, and $\mf$ and $\mf'$ constructible complexes on $U_{\et}$. 
\begin{enumerate}
\item(\cite[Lemma 2.4]{K}).
Having the same wild ramification is preserved by pullback, that is, 
if we have a commutative diagram 
\bx{U'\ar[r]^h\ar[d]&U\ar[d]\\S'\ar[r]&S,
}\ex
of excellent noetherian schemes with $U'\to S'$ separated of finite type 
and if $\mf$ and $\mf'$ have the same wild ramification over $S$, 
then $h^*\mf$ and $h^*\mf'$ have the same wild ramification over $S'$. 
\item(\cite[Lemma 3.9]{K}).
Assume that there exists $G$-torsor $V\to U$ for a finite group $G$ such that $\mc H^q(\mf)|_V$ and $\mc H^q(\mf')|_V$ are constant for every $q$. 
For an element $\sigma\in G$, we denote the quotient $V/\langle\sigma\rangle$ by $V_\sigma$. 
Then $\mf$ and $\mf'$ have the same wild ramification over $S$ 
if and only if $\mf|_{V_\sigma}$ and $\mf'|_{V_\sigma}$ have the same wild ramification over $S$ for every element $\sigma\in G$ of prime-power order.
\end{enumerate}
\end{lem}

\subsection{The notion of having universally the same conductors}\label{subsec sc}
For two constructible complexes having universally the same conductors 
means, roughly speaking, that the two complexes have the same Artin conductors after restricting to any curve. 
When we talk about ``having universally the same conductors'', we work over a base scheme 
with the property that every closed point has perfect residue field. 
This assumption on the base is made just to work with the classical Artin conductor 
and can be removed using Abbes-Saito's theory \cite{AS}. 

For a henselian trait $T$
with generic point $\eta$ and closed point $t$ with algebraically closed residue 
and for a constructible complex $\mathcal{F}$ of $\Lambda$-modules on $T$,
the Artin conductor $a(\mathcal{F})$ is defined by
$a(\mathcal{F})=\rk(\mathcal{F}_{\bar\eta})-\rk(\mathcal{F}_{t})+\Sw(\mathcal{F}_{\bar\eta})$. 
For the definition of the Swan conductor $\Sw(\mf_{\bar\eta})$, see \cite[19.3]{Se}. 

Let $X$ be a regular scheme of dimension one whose closed points have perfect residue fields 
and let $\mathcal{F}$ be a constructible complex of $\Lambda$-modules on $X$. 
For a geometric point $x$ over a closed point of $X$, 
the Artin conductor $a_x(\mathcal{F})$ at $x$ is defined by $a_x(\mathcal{F})=a(\mathcal{F}|_{X_{(x)}})$.

Let $X$ be an integral $S$-scheme separated of finite type. 
We say $X$ is an $S$-{\it curve} 
if $X$ has a compactification over $S$ which is of dimension 1. 

\begin{defn}[{\cite[Definition 2.5]{K}}]
\label{sc}
Let $S$ be an excellent noetherian scheme such that the residue field of every closed point is perfect, 
and $U$ an $S$-scheme separated of finite type. 
Let $\mf$ and $\mf'$ be constructible complexes of $\Lambda$-modules and $\Lambda'$-modules respectively on $U$, 
for finite fields $\Lambda$ and $\Lambda'$ of characteristics invertible on $S$. 

We say $\mf$ and $\mf'$ {\it have universally the same conductors over} $S$ 
if, 
for every morphism $g:C\to U$ from a regular $S$-curve $C$ 
and for every geometric point $v$ over a closed point of a regular compactification $\overline C$ of $C$ over $S$, 
we have an equality 
$a_v(j_!g^*\mf)=a_v(j_!g^*\mf')$, 
where $j$ denotes the open immersion $C\to \overline C$. 
\end{defn}

\subsection{Relation between the two notions}\label{relation}
We can easily see that having the same wild ramification implies having universally the same conductors. 
The main theorem of \cite{K} states that the converse holds if $U$ is a ``surface'':
\begin{thm}[{\cite[Theorem 3.2]{K}}]
\label{swr sc}
Let $S$ be an excellent noetherian scheme such that the residue field of every closed point is perfect. 
Let $U$ be an $S$-scheme separated of finite type which have a compactification over $S$ of dimension $\le2$. 
Let $\mf$ and $\mf'$ be constructible complexes of $\Lambda$-modules and $\Lambda'$-modules respectively on $U_{\et}$, 
for finite fields $\Lambda$ and $\Lambda'$ of characteristics invertible on $U$. 
Then, the following are equivalent;
\begin{enumerate}
\item
$\mf$ and $\mf'$ have 
the same wild ramification over $S$
\item
$\mf$ and $\mf'$ have universally the same conductors over $S$.
\end{enumerate}
\end{thm}

We prove the following refinement of Theorem \ref{swr sc} in the case where $U$ and $S$ are algebraic varieties. 
\begin{thm}\label{swr sc var}
Let $S$ be a scheme of finite type over a perfect field of characteristic $p$ 
and $U\to S$ be a morphism separated of finite type. 
Let $\mf$ and $\mf'$ be constructible complex of $\Lambda$-modules and $\Lambda'$-modules on $U_{\et}$ respectively,
for finite fields $\Lambda$ and $\Lambda'$ of characteristic distinct from $p$. 
Then the following are equivalent;
\begin{enumerate}
\item
$\mf$ and $\mf'$ have the same wild ramification over $S$,
\item
$\mf$ and $\mf'$ have universally the same conductors over $S$. 
\end{enumerate}
\end{thm}

We recall that in \cite{K}, Theorem \ref{swr sc} is reduced to Lemma \ref{reg case} below 
using resolution of singularities, for which we need the assumption on dimension. 
\begin{lem}[{\cite[Lemma 3.7]{K}}]
\label{reg case}
Under the notation in Definition \ref{sc}, assume that 
there exists $\Z/p^e\Z$-torsor $V\to U$ for some $e\ge0$ 
such that $\mc H^q(\mf)|_V$ and $\mc H^q(\mf')|_V$ are constant for all $q$. 
Further we assume that $U$ admits a regular compactification over $S$. 
Then, $\mf$ and $\mf'$ have the same wild ramification over $S$ if they have universally the same conductors over $S$. 
\end{lem}

We reduce Theorem \ref{swr sc var} to Lemma \ref{reg case} using instead purely inseparable local uniformization (Theorem \ref{temkin}) due to Temkin. 

\begin{defn}\label{local unif}
Let $X$ be an integral noetherian scheme. 
\begin{enumerate}
\item
A valuation ring $R$ with field of fractions $K(X)$ 
is {\it centered on }$X$ 
if the natural morphism $\spec K(X)\to X$ factors through $\spec K(X)\to \spec R$.
\item
Let $Y_1,\ldots,Y_m$ be integral schemes which are separated of finite type over $X$. 
We say $Y=\coprod_{i=1}^mY_i\to X$ is an $h$-{\it covering} 
if each $Y_i\to X$ is a generically finite dominant morphism 
and if, 
for every valuation ring $R$ with field of fractions $K(X)$ centered on $X$, 
there exists a valuation ring with field of fractions $K(Y_i)$ for some $i$ 
which is centered on $Y_i$ and dominates $R$. 
\item
We say an $h$-covering $Y=\coprod_iY_i\to X$ is {\it generically purely inseparable} 
if the field extensions $K(Y_i)/K(X)$ are purely inseparable. 
\end{enumerate}
\end{defn}

\begin{thm}[{\cite[Corollary 1.3.3]{Temkin}}]\label{temkin}
Let $X$ be an integral scheme of finite type over a field $k$. 
Then there exists a generically purely inseparable $h$-covering $Y\to X$ with $Y$ regular. 
\end{thm}

\begin{lem}\label{h descent}
Let $S$ be an excellent noetherian scheme characteristic $p>0$, 
$U$ an $S$-scheme separated of finite type, 
and $X$ a compactification of $U$ over $S$. 
Let $\Lambda$ and $\Lambda'$ be finite fields of characteristic distinct from $p$. 
Let $\mf$ and $\mf'$ be constructible complex of $\Lambda$-modules and $\Lambda'$-modules on $U_{\et}$ respectively 
such that $\mc H^q(\mf)$ and $\mc H^q(\mf')$ are locally constant for every $q$. 
Let $f:Y=\coprod_{i=1}^mY_i\to X$ be a generically purely inseparable $h$-covering. 
Then the following are equivalent;
\begin{enumerate}
\item
$\mf$ and $\mf'$ have the same wild ramification over $S$, 
\item
the pullbacks $f_U^*\mf$ and $f_U^*\mf'$ have the same wild ramification over $S$, 
where $f_U:Y\times_XU\to U$ is the base change of $f$. 
\end{enumerate}
\end{lem}

\begin{proof}
Since having the same wild ramification is preserved by pullback (Lemma \ref{list swr}.1), the implication $1\Rightarrow2$ follows. 
We prove the converse $2\Rightarrow 1$. We assume that 
the pullbacks $f_U^*\mf$ and $f_U^*\mf'$ have the same wild ramification over $S$. 
We show that $\mf$ and $\mf'$ satisfy the condition 3 in Lemma \ref{val criterion}. 
Let 
\bx{
U\ar[d]&\spec F\ar[l]\ar[d]\\
S&\spec\mc O_F,\ar[l]
}\ex
be a diagram with $\mc O_F$ a strictly henselian valuation ring 
and $F$ the field of fractions of $\mc O_F$. 
We assume that the image of $\spec F\to U$ is the generic point. 
By the valuative criterion of proper morphisms \cite[Th\'eor\`eme 7.3.8]{ega2}, there exists 
a unique $S$-morphism $\spec\mc O_F\to X$ extending $U\to \spec F$. 
Then, the subring $R=K(X)\cap \mc O_F$ of $K(X)$ is a valuation ring with field of fractions $K(X)$ 
and is centered on $X$. 
By the definition of $h$-coverings, there exists a valuation ring $R'$ of $K(Y_i)$ for some $i$ 
which is centered on $Y_i$ and dominates $R$. 
We denote the residue field of $F$ (resp.\ $R$) by $\kappa_F$ (resp.\ $\kappa_R$). 
We may assume that $\mc O_F$ dominates $R$ and that $\mc O_F$ is the strict henselization of the valuation ring $R$ along the inclusion $\kappa_R\to\kappa_F$. 
We take an algebraic closure $\kappa_F^{\rm alg}$ of $\kappa_F$ 
and an embedding of the residue field $\kappa_{R'}$ of $R'$ into $\kappa_F^{\rm alg}$. 
Let $\mc O_E$ be the strict henselization of $R'$ along the embedding $\kappa_{R'}\to\kappa_F^{\rm alg}$ 
and $E$ be its field of fractions. 
Since $E$ is purely inseparable over $F$, the map $G_E\to G_F$ of the absolute Galois group is bijective. 
Thus, the assertion follows. 
\end{proof}

\begin{proof}[Proof of Theorem \ref{swr sc var}]
By devissage, we may assume that 
$U$ is regular and connected 
and that 
$\mc H^q(\mf)$ and $\mc H^q(\mf')$ are locally constant for every $q$. 
We take a $G$-torsor $V\to U$ for some finite group $G$ 
such that the pullbacks $\mc H^q(\mf)|_V$ and $\mc H^q(\mf')|_V$ are constant for every $q$. 
By Lemma \ref{list swr}.2, we may assume that $G\cong\Z/p^e\Z$ for some $e\ge0$. 

We take a compactification $X$ of $U$ over $S$. 
By Theorem \ref{temkin}, we can take a generically purely inseparable $h$-covering $Y\to X$ with $Y$ regular, 
and hence, by Lemma \ref{h descent}, we may assume that $X$ is regular. 
Then the assertion follows from Lemma \ref{reg case}. 
\end{proof}

\section{Wild ramification and nearby cycle complex}
\label{swr section}
We deduce one of our main theorems in Theorem \ref{rpsi have swr} 
from the intertwining formula for $\trd$ (Lemma \ref{int trd}) 
and wildness of the terms contributing the formula (Corollary \ref{str hens}). 

\begin{lem}\label{int trd}
We use the same notations as in Lemma \ref{intertwine}: 
\StraitXlocal. 
\UopenVgaloisLambda. 
Let $\mg$ be a constructible complex of $\Lambda$-modules on $U_{\text{\'et}}$ 
such that each $\mh^q(\mg|_V)$ is constant. 
We define a virtual representation $M$ of $G$ by 
$
M=\sum_q(-1)^q\Gamma(V,\mh^q(\mg))
$. 
Then, for every $\sigma\in P_K$, 
we have an equality 
\ben\label{int trd eq}
\trd(\sigma,(R\psi j_!\mg)_x)
=\frac{1}{|G|}\sum_{g\in G}
\tr((g,\sigma),(R\psi(j_!h_{*}\ql))_x)
\cdot
\trd(g,M).
\een
\end{lem}
\begin{proof}
Since $\tr((g,\sigma),(R\psi(j_!h_{*}\ql))_x)$ is an integer by Proposition \ref{general formulation}, 
the assertion follows from Lemma \ref{intertwine}. 
\end{proof}

\begin{thm}
\label{rpsi have swr}
\StraitXlocal. 
Let $\mf$ and $\mf'$ be constructible complexes 
of $\Lambda$-modules and $\Lambda'$-modules respectively on $X_\eta$, 
for finite fields $\Lambda$ and $\Lambda'$ of characteristics invertible on $S$. 
We assume 
that $\mathcal F$ and $\mathcal F'$ have the same wild ramification over $X$. 
Then the stalks $(R\psi\mathcal F)_x$ and $(R\psi\mathcal F')_x$ of the nearby cycle complexes 
have the same wild ramification, 
that is, for every $\sigma\in P_K$, we have 
\be
\trd(\sigma,(R\psi\mf)_x)
=
\trd(\sigma,(R\psi\mf')_x).
\ee 
\end{thm}



\begin{proof}
By devissage using the induction on $\dim X_\eta$, we may assume that 
$\mf$ and $\mf'$ are of the form 
$\mathcal F\simeq j_!\mathcal G$ and $\mathcal F'\simeq j_!\mathcal G'$ 
for some dense open immersion $j:U\to X_\eta$ with $U$ being normal and connected 
and for 
some constructible complexes $\mathcal G$ and $\mathcal G'$ on $U$ 
such that $\mh^q(\mg)$ and $\mh^q(\mg')$ are locally constant. 

Then the intertwining formula (\ref{int trd eq}) for $\mg$ (resp.\ $\mg'$) holds. 
In the following we use the notation in Lemma \ref{int trd}. 
By the assumption for $\mg$ and $\mg'$ having the same wild ramification, 
we can find a normal compactification $X'$ of $U$ over $X$ 
such that 
for every $g\in G$ wildly ramified on $X'$, 
we have 
$
\trd(g,M)
=
\trd(g,M').
$
Here, if 
$\tr((g,\sigma),(R\psi(j_!h_{*}\ql))_x)$ 
is nonzero, then $g$ is wildly ramified on $X'$ 
by Corollary \ref{str hens}. 
Thus, the assertion follows from the intertwining formulas for $\mg$ and $\mg'$. 
\end{proof}

We can consider the following variant of Theorem \ref{rpsi have swr}; 
\begin{conj}\label{psi sc}
Under the notation in Theorem \ref{rpsi have swr}, 
we assume that the residue field of $S$ is algebraically closed 
and that $\mathcal F$ and $\mathcal F'$ have universally the same conductors over $X$. 
Then the stalks $(R\psi\mathcal F)_x$ and $(R\psi\mathcal F')_x$ of the nearby cycle complexes 
have the same wild ramification in the sense in Theorem \ref{rpsi have swr}. 
\end{conj}

\begin{cor}\label{psi sc curve}
Conjecture \ref{psi sc} holds if one of the following is satisfied; 
\begin{enumerate}
\item
$\dim X\le2$,
\item
$S$ is the strict localization of a smooth curve over an algebraically closed field at a closed point.  
\end{enumerate} 
\end{cor}
\begin{proof}
The case where the condition 1 (resp. 2) is satisfied follows from Theorem \ref{swr sc} (resp.\ Theorem \ref{swr sc var}) and Theorem \ref{rpsi have swr}. 
\end{proof}

\begin{rem}
\label{unreasonable}
One may think that 
it is more natural to define ``having the same wild ramification'' 
without taking blowup, that is, 
in the notations of the proof of Theorem \ref{rpsi have swr}, 
to require that $\mg$ and $\mg'$ satisfy the following property; 
for every $g\in G$ which is wildly ramified on $X$, 
we have $\trd(g,M)=\trd(g,M')$. 
But, this definition is unreasonably strong. 
Firstly, this naive definition is stronger than our definition, 
because, for a compactification $X'$ of $U$ over $X$, 
if $g\in G$ is ramified on $X'$, then it is ramified on $X$, 
but the converse does not necessarily hold. 
Further, there exists 
two sheaves which should obviously have the same wild ramification, 
but do not in the naive sense 
(see Section \ref{ex}).
\end{rem}

\section{Wild ramification and characteristic cycle}
\label{cor}
We briefly recall the definition of the characteristic cycle of a constructible \'etale sheaf due to Saito. 

Let $X$ be a smooth variety over a perfect field $k$ of characteristic $p$ 
and $\mf$ be a constructible complex of $\Lambda$-modules on $X_{\et}$ 
for a finite field $\Lambda$ of characteristic distinct from $p$. 
Beilinson defined a closed conical subset $\SS(\mf)$ of $T^*X$, called the singular support of $\mf$, 
and prove that if $X$ is pure of dimension $n$, then so is $\SS(\mf)$ (\cite{B}). 
Here, for a subset $C$ of a vector bundle $V$ on a scheme, we say $C$ is {\it conical} if it is stable under the action of the multiplicative group $\G_m$ on $V$. 

Let $f:X\to Y$ be a morphism of schemes with $Y$ a regular noetherian scheme of dimension $1$ 
and $x$ a geometric point of $X$ with image $y$ in $Y$ lying over a closed point of $Y$. 
Let $\mf$ be a complex on $X_{\et}$. 
We denote by $R\psi_x(\mf,f)$ (resp.\ $R\phi_x(\mf,f)$) 
the stalk $(R\psi\mf)_x$ (resp.\ $(R\phi\mf)_x$) 
of the nearby cycle complex $R\psi\mf$ (resp.\ vanishing cycle complex $R\phi\mf$) 
with respect to the morphism $X\times_YY_{(y)}\to Y_{(y)}$. 

The following theorem with $\Z$-linear replaced by $\Z[1/p]$-linear is proved by Saito in \cite[Theorem 5.9]{S} 
and the integrality of the coefficients is proved by Beilinson \cite[Theorem 5.18]{S}. 
\begin{thm}[{\cite[Theorem 5.9, 5.18]{S}}]\label{cc}
Let $X$ be a smooth variety of pure of dimension $n$ over a perfect field $k$. 
Let $\Lambda$ be a finite field of characteristic invertible on $k$ 
and $\mf$ be a constructible complex of $\Lambda$-modules on $X_{\et}$. 
We write $\SS(\mf)$ as the union of its irreducible components; 
$\SS(\mf)=\bigcup_aC_a$. 
We take a closed conical subset $C$ of $T^*X$ 
which is pure of dimension $n$ and contains $\SS(\mf)$. 
Then there exists a unique $\Z$-linear combination 
$A=\sum_am_aC_a$ 
satisfying the following property: 
Let $j:W\to X$ be an \'etale morphism, 
$f:W\to Y$ be a morphism to a smooth curve $Y$, 
and $u\in W$ be an at most isolated $C$-characteristic point of $f$. 
Then we have an equality 
\begin{eqnarray*}
-\dimtot_yR\phi_u(j^*\mf,f)
=
(A,df)_{T^*W,u},
\end{eqnarray*}
where the right hand side is the intersection multiplicity at the point over $u$. 
\end{thm}
See \cite[Definition 5.3]{S} for the definitions of at most isolated $C$-characteristic points and the intersection multiplicity. 
We call the linear combination $A$ in the theorem the characteristic cycle of $\mf$ and denote it by $\CC(\mf)$.

\begin{thm}[{c.f.\ \cite[Theorem 0.1]{SY}}]
\label{swr implies same cc}
Let $X$ be a smooth variety over a perfect field $k$. 
Let $\mf$ and $\mf'$ be constructible complexes of $\Lambda$-modules and $\Lambda'$-modules respectively on $X$, 
for finite fields $\Lambda$ and $\Lambda'$ of characteristics invertible in $k$. 
We take a geometric point $x$ of $X$. 
We assume 
that $\mathcal F|_{X_{(x)}}$ and $\mathcal F'|_{X_{(x)}}$ have the same wild ramification over $X_{(x)}$. 
Then, there exists an open neighborhood $U$ of the underlying point of $x$ 
such that the characteristic cycles of $\mathcal F|_U$ and $\mathcal F'|_U$ are the same: 
$CC(\mathcal F|_U)=CC(\mathcal F'|_U)$. 
\end{thm}

\begin{proof}
We may assume $X$ is equidimensional and $k$ is algebraically closed. 
Let $n=\dim X$. 
Note that, 
by a limit argument,
the assumption implies that 
there exists an \'etale neighborhood $X'$ of $x$ such that 
$\mf|_{X'}$ and $\mf'|_{X'}$ have the same wild ramification over $X'$. 
Since the problem is \'etale local, 
it suffices to show, 
under the assumption that $\mf$ and $\mf'$ have the same wild ramification over $X$, 
that the characteristic cycles of $\mathcal F$ and $\mathcal F'$ are the same. 

We take a closed conical subset $C$ of the cotangent bundle $T^*X$ 
which is purely of dimension $n$ 
and contains $\SS(\mf)\cup\SS(\mf')$. 
By the definition of the characteristic cycle 
it suffices to show that 
for every \'etale morphism $j:W\to X$, 
every morphism $f:W\to Y$ to a smooth curve $Y$, 
and every at most isolated $C$-characteristic point $u\in W$ of $f$, 
we have 
\be
\mathop{\mathrm{dimtot}}\nolimits R\phi_u(j^*\mathcal F,f)
=
\mathop{\mathrm{dimtot}}\nolimits R\phi_u(j^*\mathcal F',f).
\ee
We have a distinguished triangle 
\be
(j^*\mathcal F)_{u}\to R\psi_u(j^*\mathcal F,f)\to R\phi_u(j^*\mathcal F,f)\to
\ee
and similar one for $\mathcal F'$. 
We note that $(j^*\mathcal F)_u$ and $(j^*\mathcal F')_u$ have the same rank by assumption 
and the inertia group acts trivially on them. 
Thus, by Theorem \ref{rpsi have swr}, 
$R\phi_u(j^*\mathcal F,f)$ and $R\phi_u(j^*\mathcal F',f)$ have the same wild ramification 
and thus we obtain the above equality of the total dimensions. 
\end{proof}
Corollary \ref{swr implies same cc} immediately implies the following description of the characteristic cycles of tame sheaves:
\begin{cor}\label{cc of tame}
Let $X$ be a smooth variety over a perfect field $k$ 
and $D$ be a divisor with simple normal crossings. 
Let $\mathcal F$ be a locally constant constructible complex of $\Lambda$-modules on the complement $U=X\setminus D$. 
We assume that $\mathcal F$ is tamely ramified on $D$. 
We write $D$ as the union of irreducible components: $D=\bigcup_{i=1}^nD_i$. 
Then, we have 
\begin{eqnarray*}
CC\mathcal F=\rk\mathcal F\cdot\sum_{I\subset\{1,\ldots,n\}}T^*_{D_I}X,
\end{eqnarray*}
where $D_I$ is the intersection $\bigcap_{i\in I}D_i$ in $X$ 
and $T^*_{D_I}X$ is the conormal bundle of $D_I$ in $X$. 
\end{cor}

\begin{rem}\label{pf of cc of tame}
Corollary \ref{cc of tame} has been already proved by Saito and Yang in different ways. 
Saito proved it in \cite[Theorem 7.14]{S} 
using the explicit description \cite[Proposition 6]{S1} of vanishing cycles. 
Yang proved the equicharacteristic case of the logarithmic version of the Milnor formula (\cite[Corollary 4.2]{Yan}), 
which is equivalent to Corollary \ref{cc of tame}. 
He proved it by deforming the variety so that we can use a global result of Vidal \cite[Corollaire 3.4]{V}. 
But, we can avoid the deformation argument, using our local result Theorem \ref{rpsi have swr} instead of Vidal's result. 
This gives a simpler proof of the logarithmic version of the Milnor formula in the equicharacteristic case. 
\end{rem}

We give a variant, due to Takeshi Saito, on a relation with restrictions to curves: 
\begin{cor}[{Saito, c.f. \cite[Corollary 4.7]{K}}]
\label{curve}
Let $X$ be a smooth variety over a perfect field $k$. 
Let $\mathcal F$ and $\mathcal F'$ be constructible complexes of $\Lambda$-modules and $\Lambda'$-modules on $X$. 
We take a geometric point $x$ over a closed point $x_0$ of $X$. 
We assume that 
for every morphism $g:C\to X$ from a smooth curve $C$ 
and for every geometric point $v$ of $C$ lying above $x$, 
we have an equality of the Artin conductors; 
$a_v(g^*\mf)=a_v(g^*\mf')$. 
Then, there exists an open neighborhood $U$ of $x_0$ 
such that 
$CC(\mathcal F|_U)=CC(\mathcal F'|_U)$. 
\end{cor}
\begin{proof}
This follows from Theorem \ref{swr implies same cc} and Theorem \ref{swr sc var}. 
\end{proof}

\section{Another proof of Corollary \ref{curve} after T.\ Saito}\label{another pf}
We give a proof of Corollary \ref{curve} which the author learned from Takeshi Saito 
before the author obtained Theorem \ref{swr sc var}. 
His proof relies on Theorem \ref{rpsi have swr} and Theorem \ref{swr sc}, but not on Theorem \ref{swr sc var}. 
He established Corollary \ref{psi sc curve} 
in the case of isolated characteristic point with respect to the singular supports of given sheaves (Proposition \ref{iso sing}), 
which is sufficient to prove Corollary \ref{curve}. 
This special case of Corollary \ref{psi sc curve} is reduced to the case of surfaces using the theory of nearby cycle complexes over a general base. 

In Subsection \ref{subsec nearby}, we recall the theory of nearby cycle complexes over a general base. 
After a preliminary subsection on ``transversality'' (Subsection \ref{subsec tr}), 
we prove Theorem \ref{curve} in Subsection \ref{subsec pf}. 




\subsection{Nearby cycle over a general base}\label{subsec nearby}
We recall the theory of nearby cycle complexes over a general base, 
for which we mainly refer to \cite[Section 1]{thom}. 

Let $f:X\to Y$ be a morphism of schemes. 
For the definition of the topos $X\vani_YY$, called the vanishing topos, 
and the natural morphism $\Psi_f:X\to X\vani_YY$ of topoi, see \cite[Section 1]{thom}. 
For a topos $\mx$ and a ring $\Lambda$, let $D^+(\mx,\Lambda)$ denote the derived category of the complexes of sheaves of $\Lambda$-modules on $\mx$. 
The nearby cycle functor 
$R\Psi_f:D^+(X,\Lambda)\to D^+(X\vani_YY,\Lambda)$ 
is defined to be the functor induced by $\Psi_f$.  

Let $x$ be a geometric point of $X$ and $y$ be the image in $Y$. 
Then we have a canonical identification 
$x\vani_YY\cong Y_{(y)}$
of topoi (\cite[1.11.1]{thom}). 
This induces a morphism of topoi;
\ben\label{local section}
\sigma_x:Y_{(y)}\cong x\vani_YY\to X\vani_YY.
\een
We recall the following description of the nearby cycle complex $R\Psi_f\mf$.
\begin{lem}[{\cite[1.12.6]{thom}}]
\label{stalk}
Let $f_{(x)}$ denote the morphism $X_{(x)}\to Y_{(y)}$ induced by $f$. 
For $\mf\in D^+(X,\Lambda)$, we have a canonical isomorphism
\be
\sigma_x^*(R\Psi_f\mf)\cong Rf_{(x)*}(\mf|_{X_{(x)}}). 
\ee
\end{lem}

%

Let 
\ben\label{car}
\begin{xy}\xymatrix{X'\ar[d]_{f'}\ar[r]^{g'}&X\ar[d]^f\\Y'\ar[r]_g&Y}\end{xy}\een
be a commutative diagram. 
We have a natural morphism 
\be
g'\vani_gg:X'\vani_{Y'}Y'\to X\vani_YY
\ee
of topoi (\cite[1.4]{thom}) 
and, for a bounded below complex $\mf$ on $X_{\et}$, 
the base change morphism
\ben\label{bc}
(g'\vani_gg)^*R\Psi_f\mf\to R\Psi_{f'}({g'}^*\mf). 
\een

\begin{defn}
Let $f:X\to Y$ be a morphism of schemes, 
$\Lambda$ be a ring, 
and $\mf\in D^+(X,\Lambda)$. 
We say 
{\it the formation of $R\Psi_f\mf$ commutes with base change} 
if for any Cartesian diagram (\ref{car}), the base change morphism (\ref{bc}) is an isomorphism. 
\end{defn}

\begin{rem}\label{const}
We give a remark on constructibility of the nearby cycle complex $R\Psi_f\mf$. 
For the definition of constructibility for a sheaf on the vanishing topos $X\vani_YY$, 
see \cite[1.6]{thom} or \cite[8.1]{Or}. 
Let us just mention that, 
if we have a constructible sheaf $\mk$ on $X\vani_YY$, 
then for every geometric point $x$ of $X$ 
the sheaf $\sigma_x^*\mk$ on $Y_{(y)}$ is constructible in the usual sense. 

Assume that $X$ and $Y$ are noetherian schemes and that $f$ is a morphism of finite type. 
Let $\mf$ be a constructible complex of $\Lambda$-modules on $X_{\et}$, 
for a finite field $\Lambda$ of characteristic invertible on $Y$. 
We note that the nearby cycle complex $R\Psi_f\mf$, 
and even $\sigma_x^*R\Psi_f\mf$, may not be constructible (see \cite[Section 11]{Or} for such an example). 
But if the formation of $R\Psi_f\mf$ commutes with base change, 
then $R\Psi_f\mf$ is constructible (\cite[8.1 and 10.5]{Or}), 
and in particular, for every geometric point $x$ of $X$, the pullback of $\sigma_x^*R\Psi_f\mf$ 
is a constructible complex on $Y_{(y)}$. 
\end{rem}

\begin{lem}\label{bc stalk}
Let 
\bx{X'\ar[d]_{f'}\ar[r]^{g'}&X\ar[d]^f\\Y'\ar[r]_g&Y}\ex
be a Cartesian diagram of schemes with $f$ being of finite type. 
Let $\mf\in D^+(X,\Lambda)$ be a complex 
and $x$ be a geometric point of $X$. 
We assume that the formation of $R\Psi_f\mf$ commutes with base change. 
Then the canonical morphism 
\be
g_{(y')}^*Rf_{(x)*}\mf\to Rf'_{(x')*}{g'}^*\mf.
\ee
is an isomorphism.
\end{lem}
\begin{proof}
By the naturality (\cite[1.12]{thom}) of the isomorphism in Lemma \ref{stalk}, we obtain a commutative diagram
\bx{
g_{(y')}^*Rf_{(x)*}\mf\ar[rr]^\alpha\ar[d]_{\cong}&&Rf'_{(x')*}{g'}^*\mf\ar[d]^{\cong}\\
g_{(y')}^*\sigma_x^*R\Psi_f\mf\ar@{=}[r]&
\sigma_{x'}^*(g'\vani_gg)^*R\Psi_f\mf\ar[r]^\beta&
\sigma_{x'}^*R\Psi_{f'}({g'}^*\mf),}\ex
where $\alpha$ is the morphism in the assertion and $\beta$ is the base change map as in (\ref{bc}), 
which is an isomorphism by the assumption that 
the formation of $R\Psi_f\mf$ commutes with base change. 
\end{proof}

\begin{prop}[{\cite[Proposition 6.1]{Or}}]\label{orgogozo}
Let $X$ and $Y$ be northerian schemes, 
$f:X\to Y$ be a morphism of finite type, 
and $\mf$ be a constructible complex of $\Lambda$-modules on $X$, 
for a finite field $\Lambda$ of characteristic invertible on $Y$. 
We assume that there exists a closed subscheme $Z\subset X$ quasi-finite over $Y$ 
such that the restriction $f|_{X\setminus Z}:X\setminus Z\to Y$ of $f$ is universally locally acyclic relatively to $\mf$. 
Then, the formation of $R\Psi_f\mf$ commutes with base change. 
\end{prop}

\subsection{Complements on transversality}\label{subsec tr}
As mentioned in the beginning of this section, a special case of Conjecture \ref{psi sc} will be proved in Proposition \ref{iso sing} in the next subsection. 
More concretely, the conjecture will be proved in the case where 
the morphism $f:X\to S$ comes from a function on a smooth variety 
which is ``transversal'' outside isolated points. 
This assumption on $f$ will be used 
to take a morphism to a surface which is good with respect to the formation of the nearby cycle complex. 
In this subsection we show the existence of such a morphism (Lemma \ref{decomp}).

We begin with recalling the definition of transversality. 
\begin{defn}[{\cite[1.2]{B}}]
\label{tr}
Let $f:X\to Y$ be a morphism of smooth schemes over a field $k$ 
and $C$ be a closed conical subset of the cotangent bundle $T^*X$ of $X$. 
\begin{enumerate}
\item
Let $x$ be a point of $X$ with image $y\in Y$. 
We say $f$ is {\it $C$-transversal} at $x$ if 
we have $df_x^{-1}(C\times_Xx)\subset \{0\}$, 
where $df_x$ denotes the $k(x)$-linear map $T^*_yY\otimes_{k(y)}k(x)\to T^*_xX$ defined by pullback of differential forms. 

We can consider the subset of $X$ consisting of points at which $f$ is not $C$-transversal. 
It is a closed subset of $X$ (\cite[1.2]{B}), 
which we call the $C$-{\it characteristic locus }of $f$. 
We often regard it as a reduced closed subscheme of $X$.  
\item
We say $f$ is {\it $C$-transversal} if 
$f$ is $C$-transversal at every point of $X$. 
\end{enumerate}
\end{defn}

\begin{defn}[{\cite[1.2]{B}}]
Let $h:W\to X$ be a morphism of smooth schemes over a perfect field $k$ 
and $C$ be a closed conical subset of the cotangent bundle $T^*X$ of $X$. 
\begin{enumerate}
\item
We denote by $h^*C$ the pullback $W\times_XC\subset W\times_XT^*X$ 
and by $K$ the kernel of the morphism $dh:W\times_XT^*X\to T^*W$ of vector bundles defined by pullback of differential forms. 
Let $x$ be a point of $X$. 
We say $h$ is {\it $C$-transversal }at $x$ if 
we have $(h^*C\cap K)\times_Xx\subset \{0\}$. 
\item
We say $h$ is {\it $C$-transversal }at $x$ if 
$h$ is $C$-transversal at every point of $X$. 
\item
We assume that $h$ is $C$-transversal. 
We define a conical subset $h^\circ C\subset T^*W$ to be the image of $h^*C$ by the morphism $dh:W\times_XT^*X\to T^*W$, 
which is a closed subset of $T^*W$ by \cite[1.2]{B}. 
\end{enumerate}
\end{defn}

\begin{lem}[{\cite[Lemma 3.9.2]{S}}]\label{bc tr}
Let 
\bx{X\ar[d]_f&W\ar[d]^g\ar[l]_h\\
Y&Z\ar[l]
}\ex
be a Cartesian diagram of smooth schemes with $f$ smooth. 
Let $C\subset T^*X$ be a closed conical subset. 
Assume that $f$ is $C$-transversal. 
Then $h$ is $C$-transversal and $g$ is $h^\circ C$-transversal.
\end{lem}

We have the following fiberwise criterion of $C$-transversality. 
\begin{lem}\label{fiberwise}
Let $f:X\to Y$ and $g:Y\to Z$ be smooth morphisms of smooth schemes over a field $k$ 
and $C$ be a conical closed subset of $T^*X$. 
Let $x$ be a closed point of $X$ with image $z\in Z$. 
We assume that $g\circ f$ is $C$-transversal at $x$. 
Then the following hold. 
\begin{enumerate}
\item
The closed immersion $i:X_z\to X$ is $C$-transversal at $x$. 
\item
The following are equivalent;
\begin{enumerate}
\item
$f$ is $C$-transversal at $x$,
\item
the base change $f_z:X_z\to Y_z$ of $f$ by the closed immersion $z\to Z$ is $i^\circ C$-transversal at $x$.
\end{enumerate}
\end{enumerate}
\end{lem}

\begin{proof}
The assertion 1 and the implication (a)$\Rightarrow$(b) in the assertion 2 follow from Lemma \ref{bc tr}. 
The converse (b)$\Rightarrow$(a) 
follows from a diagram chasing 
on the commutative diagram of $k(x)$-vector spaces
\bx{
0\ar[r]&T^*_zZ\otimes_{k(z)}k(x)\ar[r]\ar@{=}[d]&T^*_xX\ar[r]&T^*_xX_z\ar[r]&0\\
0\ar[r]&T^*_zZ\otimes_{k(z)}k(x)\ar[r]&T^*_yY\otimes_{k(y)}k(x)\ar[u]\ar[r]&T^*_yY_z\otimes_{k(y)}k(x)\ar[u]\ar[r]&0}\ex
with horizontal sequences being exact, 
where $y$ denotes the image of $x$ in $Y$.  
\end{proof}

The key step in the proof of Lemma \ref{decomp} 
is to find a good function on a fiber. 
It is achieved by taking the morphism defined by a Lefschetz pencil which is enough general. 
We recall the definition and a lemma on $C$-transversality of morphisms defined by Lefschetz pencils (Lemma \ref{Legendre}). 

Let $k$ be a field and $\pp=\pp(E^\vee)=\proj_kS^\bullet E$ be the projective space 
for a finitely dimensional $k$-vector space $E$. 
Let $\pp^\vee=\pp(E)$ be the dual projective space, 
which parameterizes hyperplanes in $\pp$. 
Let $Q\subset \pp\times_k\pp^\vee$ be the universal family of hyperplanes in $\pp$, 
which parameterizes pairs $(x,H)$ of a point $x$ of $\pp$ and a hyperplane $H$ in $\pp$ with $x\in H$. 

Recall that we have canonical identifications $\pp(T^*\pp)\cong Q\cong\pp(T^*\pp^\vee)$ 
which are compatible with the projection to $\pp$ and the projection to $\pp^\vee$ respectively. 
These identifications are called the Legendre transform identifications \cite[1.5]{B}. 


Let $L\subset \pp^\vee$ be a Lefschetz pencil of hyperplanes in $\pp$, that is, a line in $\pp^\vee$. 
We define a morphism $p_L:\pp_L\to L$ by the following Cartesian diagram
\bx{
Q\ar[d]&\pp_L\ar[l]\ar[d]^{p_L}\\
\pp^\vee&L.\ar[l]
}\ex
This $p_L$ is called the morphism  defined by the Lefschetz pencil $L$. 
For a line $L\subset \pp^\vee$, 
let $A_L\subset\pp$ denote the axis of $L$, i.e, the intersection $\bigcap_{t\in L}H_t$ of hyperplanes belonging to $L$ 
Since the canonical morphism $\pp_L\to \pp$ is an isomorphism over $\pp\setminus A_L$, 
we can regard $\pp\setminus A_L$ as an open subscheme of $\pp_L$ 
and we get an induced morphism $p_L^\circ:\pp\setminus A_L\to L$. 

We recall that $C$-transverality of morphisms defined by Lefschetz pencils can be characterized using the Legendre transform identification $\pp(T^*\pp)\cong Q$: 
\begin{lem}[{\cite[Lemma 2.1]{SY}}]\label{Legendre}
Let $x\in\pp(k)$ be a rational point. 
Let $C\subset T^*_x\pp$ be a conical closed subset. 
We regard $\pp(C)$ as a closed subset of $\pp^\vee$ via the identification $\pp(T^*_x\pp)\cong Q\times_\pp x\subset \pp^\vee$. 
Then, for a line $L\subset\pp^\vee$ with the axis not containing $x$, 
the morphism $p_L^\circ:\pp\setminus A_L\to L$ is $C$-transversal (at $x$) if and only if $\pp(C)\cap L=\emptyset$. 
\end{lem}


%
We use the following elementary lemma in the proof of Lemma \ref{function} below. 
\begin{lem}\label{proj geom}
Let $\pp=\pp(E^\vee)$ be the projective space for a $k$-vector space $E$ 
and $\pp^\vee=\pp(E)$ its dual. 
Let $\G=\gr(1,\pp^\vee)$ be the Grassmanian variety parameterizing lines in $\pp^\vee=\pp(E)$. 
Let $B\subset \pp$ be a subset containing at least two closed points. 
Then 
lines $L\subset\pp^\vee$ such that any hyperplane $H\subset\pp$ belonging to $L$ does not contain $B$ 
form a dense open subset of $\G$. 
\end{lem}
\begin{proof}
%
By replacing $B$ by the minimum linear subspace containing $B$, 
we may assume that $B$ is a linear subspace of $\pp$ of dimension $d\ge1$. 
Note that a line $L\subset \pp^\vee$ satisfies the condition in the assertion if and only if $L$ does not meet 
the dual subspace $B^\vee\subset\pp^\vee$ of $B$. 
Since $B^\vee$ is of codimension $d+1\ge2$, the assertion follows. 
\end{proof}

%
%

\begin{lem}\label{function}
Let $X$ be a smooth scheme over an algebraically closed field $k$ purely of dimension $d$ 
and $C$ be a conical closed subset of the cotangent bundle $T^*X$ of dimension $\le d+1$. 
Then, locally on $X$, there exists a smooth $k$-morphism $g:X\to \A^1_k$ 
whose $C$-characteristic locus is quasi-finite over $\A^1_k$. 
\end{lem}
\begin{proof}
Since the problem is local, we may assume that $X$ is affine. 
We take a closed immersion $X\to \A^n_k$ and denote the composite $X\to \A^n_k\subset\pp^n_k$ by $i$.  
By replacing $X$ by $\pp^n_k$ and $C$ by the closure of the image of $di^{-1}(C)\subset X\times_{\pp}T^*\pp$ by $X\times_{\pp}T^*\pp\to T^*\pp$, 
we may assume that $X=\pp=\pp(E^\vee)$ for a finitely dimensional $k$-vector space $E$. 
Let $\pp^\vee$ be the dual projective space and $\G=\gr(1,\pp^\vee)$ be the Grassmanian variety parameterizing lines in $\pp^\vee$. 

Let $x$ be a closed point of $X=\pp$. 
Let $U_1$ be the dense open subset of $\G$ consisting of lines $L\subset\pp^\vee$ with the axis $A_L\subset\pp$ not containing $x$. 

We write $C$ as the union of its irreducible components; $C=\bigcup_{a\in A}C_a$. 
For each $a\in A$, let $B_a\subset \pp$ be the base of the conical closed subset $C_a$, 
i.e, $B_a=s^{-1}(C_a)$, where $s$ is the zero section $\pp\to T^*\pp$. 

Let $A_1\subset A$ be the subset consisting of irreducible components $C_a$ of $C$ 
such that the fiber $C_a\times_{\pp}x$ is not the whole cotangent space $T^*_x\pp$. 
Let $C'$ be the union $\bigcup_{a\in A_1}C_a$ 
and $C'_x\subset T^*_x\pp$ be its fiber $C'\times_{\pp}x$ at $x$. 
Let $\pp(C'_x)\subset \pp(T^*_x\pp)$ be the projectivization of $C'_x$. 
We regard $\pp(C'_x)$ as a closed subvariety of $\pp^\vee$ via the identification $\pp(T^*_x\pp)\cong Q\times_\pp x\subset \pp^\vee$. 
Since it is of codimension $\ge2$ in $\pp^\vee$, 
lines $L\subset\pp^\vee$ which do not meet $\pp(C'_x)$ form a dense open subset $U_2\subset \G$. 

Let $A_2\subset A$ be the subset consisting of irreducible components $C_a$ of $C$ 
such that the fiber $C_a\times_{\pp}x$ is the whole cotangent space $T^*_x\pp$ 
and that the base $B_a$ is of dimension $1$. 
Note that if 
$C_a\times_{\pp}x=T^*_x\pp$, 
then $\dim B_a\le1$ by the assumption that $\dim C\le d+1$. 
Applying Lemma \ref{proj geom} to $B_a$ for each $a\in A_2$, 
we get a dense open subset $U_a\subset \G$ 
consisting of lines $L\subset\pp^\vee$ such that any hyperplane $H\subset\pp$ belonging to $L$ does not contain $B_a$. 
Let $U_3$ be the intersection $\bigcap_{a\in A_2}U_a$. 

Let $L$ be a closed point of the intersection $U_1\cap U_2\cap U_3$. 
We claim that the morphism $p_L^\circ:\pp\setminus A_L\to L$ defined by $L$ produces a function with the desired properties, 
that is, if we choose a dense open immersion $\A^1\to L$ with $0\mapsto p_L(x)$, 
then the base change $(\pp\setminus A_L)\times_L\A^1\to \A^1$ satisfies the desired properties on a neighborhood of $x$. 

By Lemma \ref{Legendre}, the morphism $p_L$ is $C'=\bigcup_{a\in A_1}C_a$-transversal at $x$. 
Thus, on a neighborhood of $x$, the morphism $p_L$ is $C$-transversal outside $\bigcup_{a\in A\setminus A_1}B_a$. 
Recall that the fiber of $B_a\cap\pp_L\to L$ over $H\in L$ is the intersection $B_a\cap H$ 
and that any hyperplane $H$ belonging to $L$ does not contain $B_a$ for $a\in A_2$. 
Further, $B_a$ is of dimension $0$ for $a\in A\setminus(A_1\cup A_2)$. 
Thus, the union $\bigcup_{a\in A\setminus A_1}B_a$ is quasi-finite over $L$, which concludes the proof.  
\end{proof}

\begin{lem}\label{decomp}
Let $f:X\to Y$ be a smooth morphism of smooth schemes over an algebraically closed field $k$ 
with $X$ and $Y$ equidimensinal of dimension $n\ge2$ and $1$ respectively. 
Let $C$ be a conical closed subset of the cotangent bundle $T^*X$ 
which is equidimensional of dimension $n$. 
We assume that the $C$-characteristic locus $Z\subset X$ of $f$ is quasi-finite over $Y$. 
Then, locally on $X$, there exists a smooth $Y$-morphism $g:X\to \A^1_Y$ 
whose $C$-characteristic locus is quasi-finite over $\A^1_Y$. 
\end{lem}
\begin{proof}
Let $x$ be a closed point of $X$ with image $y\in Y$. 
Let $i$ denote the closed immersion $X_y\to X$. 
We define $C_0$ to be the closure of the image $di(i^*C)$ in the cotangent bundle $T^*X_y$ of $X_y$. 
Note that $C_0$ is of dimension $\le n$ and that $X_y$ is equidimensional of dimension $n-1$. 

By applying Lemma \ref{function} to the smooth scheme $X_y$ and the closed conical subset $C_0$, 
we can find, 
after replacing $X$ by an open neighborhood of $x$ if needed, 
a smooth morphism $g_0:X_y\to\A^1_y$ 
whose $C_0$-chracteristic locus $W_0\subset  X_y$ is quasi-finite over $\A^1_y$. 
We take a $Y$-morphism $g:X\to\A^1_Y$ inducing $g_0$ by base change. 

Then $g$ is flat, and hence smooth, on a neighborhood of $X_y$, by a local criterion of flatness, \cite[Corollaire 5.9]{sga1}. 
Further, by Lemma \ref{fiberwise}, 
the $C$-characteristic locus $W$ of $g$ satisfies $W\times_Yy= W_0\cup Z_y$. 
Since $W_0\cup Z_y$ is quasi-finite over $\A^1_y$, the closed subset $W$ is quasi-finite over an open neighborhood of $y\in Y$, 
which concludes the proof. 
\end{proof}

\subsection{Proof of Corollary \ref{curve}}\label{subsec pf}
We prove a special case of Conjecture \ref{psi sc} assuming the case 1 of Corollary \ref{psi sc curve}, but not the case 2. 

\begin{lem}\label{orgogozo2}
Let $f:X\to S$ be a morphism of smooth schemes over a perfect field $k$. 
Let $x$ be a geometric point of $X$ lying above a closed point. 
Let $\mf$ be a constructible complex of $\Lambda$-modules on $X$, 
for a finite field $\Lambda$ of characteristics invertible in $k$. 
We assume that the $\SS(\mf)$-characteristic locus (Definition \ref{tr}.1) of $f$ is quasi-finite over $S$. 
Then 
\begin{enumerate}
\item
the formation of $R\Psi_f\mf$ commutes with base change, 
\item
$\sigma_x^*R\Psi_f\mf\cong Rf_{(x)*}\mf$ is a constructible complex on $Y_{(y)}$. 
\end{enumerate}
\end{lem}
\begin{proof}
1. 
Since $f$ is universally locally acyclic outside the $\SS(\mf)$-characteristic locus 
by the definition of singular support \cite[1.3]{B}, 
the assertion follows from Proposition \ref{orgogozo}. 

2. 
Follows from the assertion 1 and Remark \ref{const}.
\end{proof}
The assertion 2 of Lemma \ref{orgogozo2} enables us to state the assertion 1 in the following proposition.

\begin{prop}[Saito]\label{iso sing}
Let $f:X\to S$ be a morphism of smooth schemes over a perfect field $k$. 
Let $x$ be a geometric point of $X$ lying above a closed point. 
Let $\mf$ and $\mf'$ be constructible complexes of $\Lambda$-modules and $\Lambda'$-modules respectively on $X$, 
for finite fields $\Lambda$ and $\Lambda'$ of characteristics invertible in $k$. 
We assume that the $\SS(\mf)\cup\SS(\mf')$-characteristic locus (Definition \ref{tr}.1) of $f$ is quasi-finite over $S$ 
and that $\mf|_{X_{(x)}}$ and $\mf'|_{X_{(x)}}$ have universally the same conductors over $X_{(x)}$. 
Then 
\begin{enumerate}
\item
$Rf_{(x)*}\mf$ and $Rf_{(x)*}\mf'$ have universally the same conductors over $S_{(s)}$, 
where $s$ is the image of $x$ in $S$,
\item
when $S$ is a smooth curve, the stalks 
$R\psi_x(\mf,f)$ and $R\psi_x(\mf',f)$ have the same wild ramification in the sense in Theorem \ref{rpsi have swr}. 
\end{enumerate}
\end{prop}


\begin{proof}
Let us first note that, when $S$ is a smooth curve,  
$R\psi_x(\mf,f)$ and $R\psi_x(\mf',f)$ have the same wild ramification 
if and only if 
$Rf_{(x)*}\mf$ and $Rf_{(x)*}\mf$ do. 
In fact, 
we have canonical isomorphisms 
\be
(Rf_{(x)*}\mf)_{\bar\eta}&\cong& R\psi_x(\mf,f),\\
(Rf_{(x)*}\mf)_{s}&\cong& \mf_x,
\ee
where $\bar\eta$ is a generic geometric point of $S_{(s)}$. 
In particular, the assertion 2 follows from the assertion 1. 
In the following we show the assertion 1. 


We may assume that $k$ is algebraically closed. 
By replacing $f:X\to S$ by the second projection $X\times_kS\to S$ and 
$\mf$ and $\mf'$ by the direct images $\Gamma_{f*}\mf$ and $\Gamma_{f*}\mf'$ 
by the graph $\Gamma_f:X\to X\times_kS$ of $f$, 
we may assume that $f$ is smooth. 
Since the assumption implies, by Lemma \ref{orgogozo2}, that formations of $R\Psi_f\mf$ and $R\Psi_f\mf'$ commute with base change, 
we can use Lemma \ref{bc stalk} to reduce the problem to the case where $S$ is a smooth curve. 

We prove the assertion for a smooth morphism $f$ to a smooth curve by induction on $\dim X$. 
The case where $\dim X\le2$ is already proved in Corollary \ref{psi sc curve}.  
In general, we take a smooth morphism to a smooth surface; 
by Lemma \ref{decomp}, we can find, shrinking $X$ if needed, 
a smooth morphism $g:X\to Y=\A^1_S$ whose $\SS(\mf)\cup\SS(\mf')$-characteristic locus $W$ is quasi-finite over $Y$. 
Then, by Lemma \ref{orgogozo2}, 
the formations of 
$R\Psi_g\mf$ and $R\Psi_g\mf'$ commute with base change. 
We claim that the complexes 
$Rg_{(x)*}\mf$ and $Rg_{(x)*}\mf'$ 
have universally the same conductors over $Y_{(y)}$, 
where $y$ is the image of $x$ in $Y$. 

Let $h:C\to Y$ be a morphism from a smooth curve $C$ 
and $v$ be a geometric point of $C$ lying over $y$. 
Let $X'$ denote the fiber product $X\times_YC$ 
and $x'$ be a geometric point of $X'$ lying above $x$ and $v$. 
We fix notations by the following Cartesian diagram; 
\ben\label{cart}\begin{xy}\xymatrix
{X'\ar[d]_{g'}\ar[r]^{h'}&X\ar[d]^g\\
C\ar[r]_h&Y}\end{xy}\een
Then by Lemma \ref{bc stalk}, we have a canonical isomorphism
\be
h_{(y')}^*Rg_{(x)*}\mf\cong Rg'_{(x')*}{h'}^*\mf.
\ee
Note that 
we have $\dim X'=\dim X-1$ since $g$ is flat. 
We also note that Lemma \ref{bc tr} implies that 
$h'$ is $\SS(\mf)\cup\SS(\mf')$-transversal 
and that 
the morphism $g':X'\to C$ is ${h'}^\circ\SS(\mf)\cup{h'}^\circ\SS(\mf')$-transversal outsider outside $W\times_YC$, which is quasi-finite over $C$. 
Since, we have, by the definition of singular support \cite[1.3]{B}, $\SS({h'}^*\mf)\subset {h'}^\circ\SS(\mf)$ and the corresponding inclusion for $\mf'$, 
the $\SS({h'}^*\mf)\cup\SS({h'}^*\mf')$-characteristic locus of $g'$ is quasi-finite over $C$. 
Thus, by the induction hypothesis, 
$Rg_{(x)*}\mf$ and $Rg_{(x)*}\mf'$
have universally the same conductors over $Y_{(y)}$. 

Since $Y_{(y)}$ is of dimension $2$, we can apply
Corollary \ref{psi sc curve} to obtain 
the assertion (using the remark at the beginning of the proof). 
\end{proof}

\begin{proof}[Proof of Corollary \ref{curve} (Saito)]
The assumption implies 
that $\mathcal F|_{X_{(x)}}$ and $\mathcal F'|_{X_{(x)}}$ have universally the same conductors over $X_{(x)}$. 
Since the problem is \'etale local, we may assume that $\mf$ and $\mf'$ have universally the same conductors over $X$. 
We take 
an \'etale morphism $j:W\to X$, 
a morphism $f:W\to Y$ to a smooth curve $Y$, 
and a point $u\in W$ such that $f$ is $C$-transversal outside $u$. 
Then, by Proposition \ref{iso sing}, 
we have 
\be
\mathop{\mathrm{dimtot}}\nolimits R\phi_u(j^*\mathcal F,f)
=
\mathop{\mathrm{dimtot}}\nolimits R\phi_u(j^*\mathcal F',f).
\ee
Thus, by the definition of characteristic cycle, we have $\CC(\mf)=\CC(\mf')$.  
\end{proof}

\appendix
\def\thesection{\Alph{section}}
\section{$\ell$-adic systems of complexes with a pro-finite group action}
\label{ell adic}
We construct in Proposition \ref{inverse limit of complexes} an ``inverse limit'' of an $\ell$-adic system of perfect complexes of continuous $G$-modules 
for an admissible $G$ relatively to $\zl$ (for the definition of admissibility, see Definition \ref{alm prime-to-ell}). 

In the following, a morphism of complexes means a morphism in the category of complexes 
(not in the homotopy category).

\begin{prop}[{c.f.\ \cite[Proposition 10.1.15]{Fu} and \cite[\S3.3]{Houzel}}]
\label{inverse limit of complexes}
Let $G$ be a pro-finite group which is admissible relatively to $\zl$ 
and $(K_n)_{n\geq1}$ be an inverse system of complexes of continuous $\zl[G]$-modules finitely generated over $\zl$ satisfying the following properties;
\begin{itemize}
\item
for each $n\geq1$, $K_n^i$ is a projective continuous $\zmodln[G]$-module, 
\item
the transition morphism $K_{n+1}\to K_n$ gives a quasi-isomorphism 
$$K_{n+1}\otimes^L_{\zz/\ell^{n+1}\zz}\zmodln\to K_n.$$ 
\end{itemize}
Then, 
\begin{enumerate}
\item
there exists a bounded complex $K$ of projective $\zl[G]$-modules finitely generated over $\zl$ 
together with a sequence of quasi-isomorphism 
$(K\ltensor_{\zl}\zmodln\to K_n)_{n\geq1}$ 
compatible with transition morphisms. 
\item
For a complex $K$ as above, we have 
\begin{enumerate}
\item
a natural isomorphism
\begin{align*}
\varprojlim_nH^i(K_n)\cong H^i(K),
\end{align*}
\item
for an element $g\in G$ with surnatural order prime-to-$\ell$, an equality
\begin{align*}
\trbr(g,K_1)
=
\sum_i(-1)^i\tr(g,H^i(K)\otimes_{\zl}\ql).
\end{align*}
\end{enumerate}
\end{enumerate}
\end{prop}

The key of the proof of Proposition \ref{inverse limit of complexes} is Lemma \ref{lift qis} below, 
which we deduce from Lemmas \ref{lift acyclic}, \ref{lift homotopic}, and \ref{lift split inj}. 
In these Lemmas, we work with 
an Artinian local ring $A$ with a proper ideal $I$ 
and a profinite group $G$ which is admissible relatively to $A$. 
We write $A_0$ for the quotient $A/I$. 
Similarly, for an $A$-module $M$, we write $M_0$ for the quotient $M/IM$.

\begin{lem}[{c.f.\ \cite[Lemma 10.1.10]{Fu} and \cite[Corollaire in \S3.3]{Houzel}}]
\label{lift acyclic}
Let $C_0$ be a bounded acyclic complex of projective continuous $A_0[G]$-modules finitely generated over $A_0$. 
Then there exists a bounded acyclic complex $C$ of projective continuous $A[G]$-modules finitely generated over $A$ with an isomorphism $C\otimes_{A}A_0\cong C_0$. 
\end{lem}

\begin{proof}
Note that the assumption that $A$ is Artinian assures that 
every projective $A_0[G]$-module finitely generated over $A_0$ 
admits 
a lift to a projective $A[G]$-module finitely generated over $A$. 
In fact, the problem can be reduced to the case where $G$ is finite (Lemma \ref{char of proj}). 
Then it follows from \cite[14.3, Proposition 41]{Se}. 

Arguing by induction on the length of the complex $C$, 
the problem is reduced to showing the following: 
Let $N$ be a projective $A[G]$-module finitely generated over $A$. 
Then every short exact sequence $0\to L_0\to M_0\to N_0\to 0$ of projective continuous $A_0[G]$-modules finitely generated over $A_0$, 
admits a lift, that is, 
a short exact sequence $0\to L\to M\to N\to 0$ of projective continuous $A[G]$-modules finitely generated over $A$ 
which lifts $0\to L_0\to M_0\to N_0\to 0$. 
For this, lift $M_0\to N_0$ to $M\to N$ and take $L$ as the kernel of the lift $M\to N$. 
\end{proof}

\begin{lem}[{c.f.\ \cite[Lemma 10.1.11]{Fu}}]
\label{lift homotopic}
Let $\phi:M\to N$ be a morphism of complexes of projective continuous $A[G]$-modules. 
Let $\psi:M_0\to N_0$ be a morphism of complexes which is homotopic to $\phi_0=\phi\otimes\id_{A_0}$. 
Then, there exists a morphism $\psi:M\to N$ which is homotopic to $\phi$ such that $\psi\otimes\id_{A_0}=\psi_0$. 
\end{lem}
\begin{proof}
For a homotopy $k_0:M_0\to N_0[-1]$ between $\phi_0$ and $\psi_0$, 
that is, a map $k_0$ satisfying $dk_0+k_0d=\phi_0-\psi_0$,  
pick a lift $k:M\to N[-1]$ of $k_0$ and define $\psi=\phi-dk-kd$. 
\end{proof}

\begin{lem}
\label{lift split inj}
For a morphism $\phi:M\to N$ of projective continuous $A[G]$-modules finitely generated over $A$, 
$\phi$ is a split injection if and only if $\phi_0=\phi\otimes\id_{A_0}$ is. 
\end{lem}
\begin{proof}
We write $M_0$ and $N_0$ for $M\otimes_AA_0$ and $N\otimes_AA_0$. 
We take a splitting $\psi_0:N_0\to M_0$ of $\phi_0$; so $\psi_0\circ\phi_0=\id$. 
By projectivity, there exists a lift $\psi$ of $\psi_0$. 
Thus, it suffices to show that lifts $M\to M$ of an automorphism $M_0\to M_0$ are automorphisms. 
But this follows from Nakayama's lemma (forget the actions of $G$). 
\end{proof}
\begin{lem}[{c.f.\ \cite[Lemma 10.1.12]{Fu} and \cite[Lemme 1 in \S 3.3]{Houzel}}]
\label{lift qis}
Let $M$ (resp.\ $N_0$) be 
a bounded complex of projective continuous $A[G]$-modules 
(resp. projective $A_0[G]$-modules) 
finitely generated over $A$ 
and let $\phi_0:M_0=M\otimes_{\zl}A_0\to N_0$ be a quasi-isomorphism. 
Then, 
there exists a bounded complex $N$ of projective continuous $A[G]$-modules finitely generated over $A$ 
with an isomorphism $N\otimes_{A}A_0\cong N_0$ 
and a quasi-isomorphism $\phi:M\to N$ such that $\phi\otimes\id_{A_0}=\phi_0$. 
\end{lem}
\begin{proof}
First, we reduce the problem to the case where 
$\phi_0^i:M_0^i\to N_0^i$ is surjective for every $i$. 
Let $C_0$ be the mapping cone of $\id_{N_0}:N_0\to N_0$. 
Then, by Lemma \ref{lift acyclic}, 
there exists a bounded complex $C$ of projective continuous $A[G]$-modules finitely generated over $A$ 
with an isomorphism $C\ltensor_{A}A_0\cong C_0$. 
By replacing $M$ by $M\oplus C$ 
and $\phi_0:M_0\to N_0$ by the map 
$M_0\oplus C_0\to N_0$ 
defined by 
$M_0^i\oplus C_0^i=M_0^i\oplus N_0^i\oplus N_0^{i+1}\to N_0^i:(x,y,z)\to\phi(x)+y$, 
we may assume that 
$\phi_0^i:M_0^i\to N_0^i$ 
is surjective for every $i$. 

By taking the kernel $K_0^i$ of the surjection 
$\phi_0^i:M_0^i\to N_0^i$, 
we obtain an acyclic complex $K_0$. 
By Lemma \ref{lift acyclic}, 
there exists a bounded acyclic complex $K$ of projective continuous $A[G]$-modules finitely generated over $A$ 
with an isomorphism $K\ltensor_{A}A_0\cong K_0$. 
Since the morphism 
$\iota_0:K_0\to M_0$ is homotopic to zero, 
we can find, using Lemma \ref{lift homotopic}, 
a morphism $\iota:K\to M$ lifting $\iota_0$. 
Further, since $\iota_0^i:K_0^i\to M_0^i$ is a split injection for every $i$, 
we can use Lemma \ref{lift split inj} to deduce that 
$\iota^i:K^i\to M^i$ is a split injection for every $i$. 
Then it suffices to take $N^i$ as the cokernel of $\iota^i$, 
which is a projective continuous $A[G]$-module 
and $\phi$ as the natural surjection.  
\end{proof}

\begin{proof}[Proof of Proposition \ref{inverse limit of complexes}]
1. It suffices to construct a sequences of quasi-isomorphism 
$(u_n:K_n\to K'_n)_{n\ge1}$ 
such that 
$(K'_n)_{n\ge1}$ is 
an inverse systems of complexes of continuous $\zl[G]$-modules finitely generated over $\zl$ 
satisfying the same properties as $(K_n)_{n\ge1}$ 
and inducing an {\it isomorphism }$K'_{n+1}\otimes_{\zz/\ell^{n+1}\zz}\zmodln\to K'_n$. 
In fact, from $(K'_n)_{n\ge1}$ we get a desired complex by taking inverse limit termwise. 

We construct 
$(u_n:K_n\to K'_n)_{n\ge1}$ 
inductively as follows. 
Put ${{K'_1}}=K_1$ and $u_1=\mathrm{id}_{K_1}$. 
We assume that $n\ge2$ and we have defined a quasi-isomorphism $u_i:K_i\to K'_i$ for $i\le n-1$; 
\begin{eqnarray*}
\begin{xy}
\xymatrix{
K_n\ar@{.>}[rr]^{u_n}\ar[d]^{\text{mod }\ell^{n-1}}&&{K'_n}\ar@{.>}[d]^{\text{mod }\ell^{n-1}}\\
K_n/\ell^{n-1}\ar[r]&K_{n-1}\ar[r]^{u_{n-1}}&{K'_{n-1}}.
}
\end{xy}
\end{eqnarray*}
We apply Lemma \ref{lift qis} to the composite 
$K_n/\ell^{n-1}\to K_{n-1}\to {K'_{n-1}}$, 
we find a bounded complex ${K'_n}$ of finite projective continuous $\zz/\ell^n\zz[G]$-modules 
with an isomorphism ${K'_n}/\ell^{n-1}\cong {K'_{n-1}}$ 
and a quasi-isomorphism $K_{n}\to {K'_n}$ whose mod $\ell^{n}$ is the quasi-isomorphism $K_n/\ell^{n-1}\to {K'_{n-1}}$. 

2. The isomorphism (a) follows from the fact that the inverse limit functor is exact on finite abelian groups. 
The equality (b) can be obtained as follows: 
\begin{align*}
\trbr(g,K\ltensor_{\zl}\fl)
&=\sum_i(-1)^i\trbr(g,K^i\otimes_{\zl}\fl)\\
&=\sum_i(-1)^i\tr(g,K^i\otimes_{\zl}\ql)\\
&=\sum_i(-1)^i\tr(g,H^i(K)\otimes_{\zl}\ql).
\end{align*}
Here the first equality is by definition and the second follows from projectivity. 
\end{proof}

\section{An example of big local monodromy}\label{ex}
We give an example of two sheaves which should obviously have the same wild ramification, 
but do not in the naive sense as in Remark \ref{unreasonable}. 
To do that we give an example showing that 
inertia groups in higher dimensional case are very big. 
The author learned the following from Takeshi Saito. 

Let $k$ be a field with a separable closure $\bar k$ 
and $C$ be a dense open subscheme of $\pp_k^1$. 
We consider the map $\mathrm{pr}:\pp^2\setminus{(0:0:1)}\to\pp^1;(x:y:z)\mapsto(x:y)$. 
We put $U=\pr\inverse (C)$ and $X=\pp^2$. 
Let $x$ be the origin $(0:0:1)$ of $\pp^2_{\bar k}$. 
The following lemma shows that the inertia group 
$\pi_1(U\times_XX_{(x)},a)$, 
for a geometric point $a$, is very big. 
\begin{lem}\label{big}
We assume that the geometric point $a$ lies above a geometric generic point $\bar\eta$ of $C_{\bar k}$. 
Then the natural morphism 
\be
\pi_1(U\times_XX_{(x)},a)
\to
\pi_1(C_{\bar k},\bar\eta)
\ee
induced by the composite 
$U\times_XX_{(x)}
\to
U_{\bar k}\to C_{\bar k}$
is surjective. 
\end{lem}
\begin{proof}
We may assume that $k$ is separably closed. 
Let $X'\to X$ be the blowup of at $x$ 
and $\xi$ be the generic point of the exceptional divisor of $X'$. 
Then, we have a natural morphism 
$X'_{(\xi)}\to X_{(x)}$. 

Let $a'$ be a geometric point of $U\times_{X'}X'_{(\xi)}$ lying above $a$. 
We consider a commutative diagram
\begin{eqnarray}\label{blowup}
\begin{xy}
\xymatrix{
\pi_1(U\times_{X'}X'_{(\xi)},a')\ar[r]\ar[d]^\alpha
&
\pi_1(U\times_{X}X_{(x)},a)\ar[r]&
\pi_1(U,a)
\ar[d]
\\
\pi_1(X'_{(\xi)},a')\ar[rr]^\beta
&
&
\pi_1(C,\bar\eta).}
\end{xy}
\end{eqnarray}
Since the composite $\xi\to X'_{(\xi)}\to \eta$ is an isomorphism, 
the homomorphism $\pi_1(X'_{(\xi)},a')\to\pi_1(\eta,\bar\eta)$ is an isomorphism. 
Thus, $\beta$ is surjective. 
Since $\alpha$ is surjective, we obtain the assertion by a diagram chasing. 
\end{proof}


We make it precise what the naive sense means. 
Let $X$ be an excellent noetherian scheme and 
$U$ a dense open subscheme of $X$. 
For an $\F_\ell$-sheaf $\mf$ and an $\F_{\ell'}$-sheaf $\mf'$ on $U_{\et}$ 
which are locally constant and constructible, 
having the same wild ramification over $X$ in the naive sense means that, 
under the notations in Definition \ref{swr}, 
for every $g\in G$ wildly ramified on $X$, we have 
$\dim_\Lambda M^g=\dim_{\Lambda'}(M')^g$. 
As mentioned in Remark \ref{unreasonable}, 
having the same wild ramification over $X$ in the naive 
is stronger than having the same wild ramification over $X$. 
Through an example we will see that it is unreasonably strong.

Assume that the characteristic $p$ of $k$ is different from $2$ and $3$. 
Let $f:E\to C=\pp_k^1\setminus\{0,1,\infty\}$ be the Legendre family of elliptic curves defined by the equation $y^2=x(x-1)(x-\lambda)$, $\lambda\in\pp^1\setminus\{0,1,\infty\}$. 
This family has big monodromy; 
for example, it has the following property.  
\begin{thm}
\label{legendre}
The monodromy representation 
$\pi_1(C_{\bar k},\bar\eta)\to SL_2(\F_\ell)$ 
associated to $R^1f_*\F_\ell$  
is surjective for any prime number $\ell\ne2,p$. 
\end{thm}

We denote the natural map $U\to C$ by $\pi$. 
We consider the sheaf 
\be
\mf_\ell=\pi^*R^1f_*\F_\ell,
\ee
which is a locally constant and constructible sheaf of rank $2$ on $U$. 
We note that 
$R^1f_*\F_\ell$ is 
tamely ramified along $0,1,\infty$. 
In particular, 
$R^1f_*\F_\ell$ and $R^1f_*\F_{\ell'}$ 
for prime numbers $\ell$ and $\ell'$ which are distinct from the characteristic of $k$ 
have the same wild ramification over $k$. 
Since having the same wild ramification is preserved by pullback \cite[Lemma 2.4]{K}, 
$\mf_\ell$ and $\mf_{\ell'}$ have the same wild ramification over $k$. 

\begin{cor}
\begin{enumerate}
\item
The monodromy representation 
$\pi_1(U\times_{X}X_{(\bar x)},a)\to SL_2(\F_\ell)$ 
associated to $\mf_\ell$ is surjective for $\ell\ne2,p$.  
\item
We assume that $k$ is of positive characteristic $p$. 
Then, if $\ell\ne2$ is a prime number such that 
$\ell\equiv\pm1\mod p$ 
and 
$\ell'\ne p$ is a prime number such that 
$\ell'\not\equiv\pm1\mod p$, 
then 
$\mf_{\ell}$ 
and 
$\mf_{\ell'}$ 
do not have the same wild ramification in the naive sense.  
\end{enumerate}
\end{cor}
\begin{proof}
1. follows from Lemma \ref{big} and Theorem \ref{legendre}. 

2. We note that for any (pointed) Galois \'etale covering of $U\times_XX_{(\bar x)}$ with Galois group $G$, 
the image of any element 
$\sigma\in\pi_1(U\times_{X}X_{(\bar x)},a)$ 
by 
$\pi_1(U\times_{X}X_{(\bar x)},a)\to G$ 
is ramified on $X_{(\bar x)}$ in the sense of Definition \ref{ramified}. 
Then the assertion follows 
since the order of the finite group $SL_2(\F_\ell)$ is $\ell(\ell-1)(\ell+1)$. 
\end{proof}



%

Further, this construction with $\F_\ell$ replaced by $\qlbar$ gives 
an example of a sheaf whose Frobenius traces are rational numbers but whose monodromy along boundary has eigenvalues which are not algebraic numbers. 

We use the following big monodromy theorem instead of Theorem \ref{legendre}. 	
\begin{thm}
\label{legendre ql}
The image of the monodromy representation $\pi_1(C_{\bar k},\bar\eta)\to SL_2(\qlbar)$ 
associated to $R^1f_*\qlbar$ is open. 
In particular, there exists $\gamma\in\pi_1(C_{\bar k},\bar\eta)$ 
such that the eigenvalues of the action of $\gamma$ on $(R^1f_*\qlbar)_{\bar\eta}$ are not algebraic numbers. 
\end{thm}
\begin{cor}
There exists an element $\sigma\in\pi_1(U\times_{X}X_{(\bar x)},a)$ 
such that the eigenvalues of the action of $\sigma$ on $(\pi^*R^1f_*\qlbar)_{a}$ are not algebraic numbers. 
\end{cor}
\begin{proof}
This follows from Theorem \ref{legendre ql} and Lemma \ref{big}. 
\end{proof}

In the terminologies of \cite{LZ}, the above construction gives a compatible system which is not compatible along a boundary. 
We assume $k$ is a finite field. 
Let $\mathbb L$ be a set of prime numbers distinct from the characteristic of $k$. 
Then, $(\mf_\ell)_{\ell\in\mathbb L}$ is $\Q$-compatible in the sense in \cite[Definition 1.14]{Z}. 
But it is not compatible on $X$ in the sense in \cite[Definition 1.1]{LZ}.

\end{document}